\documentclass{amsart}
\newtheorem{thm}{Theorem}[section]
\newtheorem{cor}[thm]{Corollary}
\newtheorem{lem}[thm]{Lemma}
\newtheorem{prop}[thm]{Proposition}
\theoremstyle{definition}
\newtheorem{defn}[thm]{Definition}
\theoremstyle{remark}
\newtheorem{rem}[thm]{Remark}
\numberwithin{equation}{section}
\newcommand{\norm}[1]{\left\Vert#1\right\Vert}
\newcommand{\abs}[1]{\left\vert#1\right\vert}
\newcommand{\set}[1]{\left\{#1\right\}}

\newcommand{\eps}{\varepsilon}

\newcommand{\dbar}{\bar\partial}
\newcommand{\ddbar}{\partial\bar\partial}
\newcommand{\dom}[1]{\mathrm{Dom}(#1)}

\newcommand{\re}{\mathrm{Re}}
\newcommand{\im}{\mathrm{Im}}
\DeclareMathOperator{\Tr}{Tr}

\newcommand{\p}{\partial}

\begin{document}

\title[Closed Range in Stein Manifolds]{Closed Range for $\dbar$ and $\dbar_b$ on Bounded Hypersurfaces in Stein Manifolds}%
\author{Phillip S. Harrington and Andrew Raich}%
\address{SCEN 301, 1 University of Arkansas, Fayetteville, AR 72701}%
\email{psharrin@uark.edu, araich@uark.edu}%

\thanks{The first author is partially supported by NSF grant DMS-1002332 and the second author is partially supported by NSF grant DMS-0855822}%

\begin{abstract}We define weak $Z(q)$, 
a generalization of $Z(q)$ on bounded domains $\Omega$ in a Stein manifold $M^n$ that suffices to prove closed range of $\dbar$.
Under the hypothesis of weak $Z(q)$, we also show (i)
that harmonic $(0,q)$-forms are trivial and (ii) if $\partial\Omega$ satisfies weak $Z(q)$ and weak $Z(n-1-q)$, then $\dbar_b$ has closed range on $(0,q)$-forms on $\p\Omega$. We provide examples to show
that our condition contains examples that are excluded from $(q-1)$-pseudoconvexity and the authors' previous notion of weak $Z(q)$.
\end{abstract}
\maketitle
\section{Introduction}
The purpose of this article is to establish sufficient conditions for the closed range of $\dbar$ (and $\dbar_b$) on not necessarily pseudoconvex domains (and their boundaries)
in Stein manifolds. We pay particular attention to keeping the boundary regularity at a minimum; our results holds for $C^3$ boundaries. In \cite{HaRa11}, we develop a notion
of weak $Z(q)$ for which we can prove closed range of $\dbar_b$ for smooth bounded CR manifolds of hypersurface type in $\mathbb C^n$. In this paper, we generalize our notion of weak $Z(q)$ and relax
the smoothness assumption. The microlocal analysis technique of \cite{HaRa11} requires significant boundary smoothness, and to replace the microlocal analysis, we assume that
our CR manifold is the boundary of a bounded domain in a Stein manifold and attain the closed range of $\dbar_b$  as a consequence of the $\dbar$-theory which we prove. Our analysis
of $\dbar_b$ is in the spirit of \cite{Shaw03}. Additionally, we show that the weak $Z(q)$-hypothesis is sufficient to show that harmonic forms vanish at level $(0,q)$. Finally, we provide examples
to show that our new condition is more general than either $(q-1)$-pseudoconvexity \cite{Zam08} or the weak $Z(q)$-hypothesis of \cite{HaRa11}.

Let $\Omega\subset\mathbb{C}^n$ be a bounded domain with a $C^2$ defining function $\rho$.  The Levi form of $\partial\Omega$ is the form
\[
  \mathcal{L}_\rho(t)=\sum_{j,k=1}^n \frac{\partial^2\rho}{\partial z_j\partial\bar{z}_k}t_j\bar{t}_k,
\text{where }
  \sum_{j=1}^n\frac{\partial\rho}{\partial z_j} t_j=0.
\]
If the Levi form is positive semi-definite for all boundary points, we say $\Omega$ is \emph{pseudoconvex}.
Suppose that $f$ is a $(0,q)$-form on $\Omega$ with components in $L^2$.  If $\dbar f=0$, where $\dbar$ is the Cauchy-Riemann operator, we wish to know whether there exists a $(0,q-1)$-form $u
$ with $L^2$ components such that $\dbar u=f$.  When $\Omega$ is a pseudoconvex domain, this question was answered in the affirmative for all $1\leq q\leq n$ by H\"ormander in \cite{Hor65}.  In
fact, pseudoconvexity is a necessary condition to solve $\dbar$ in $L^2$ for all $1\leq q\leq n$.

If we only wish to solve $\dbar$ for a fixed value of $q$, pseudoconvexity is no longer necessary.  If the Levi form has at least $n-q$ positive eigenvalues or at least $q+1$ negative eigenvalues, we say that
$\partial\Omega$ satisfies $Z(q)$.  It is known that the $\dbar$ problem can be solved in $L^2$ if $\partial\Omega$ satisfies $Z(q)$. In fact, $Z(q)$ is equivalent to the solvability
of $\dbar u=f$ if the components of $u$ are required to be elements of the $L^2$ Sobolev space $W^{1/2}$ (see \cite{Hor65}, Theorem 3.2.2 in \cite{FoKo72}, or \cite{AnGr62}).

If we allow the Levi form to degenerate as in the pseudoconvex case, solvability is less well understood.  One candidate that has been studied in some detail is $q$-pseudoconvexity (see \cite{Zam08}).  In
$q$-pseudoconvexity, a subbundle of the tangent bundle of at most dimension $q$ exists locally with the property that the sum of the $q+1$ smallest eigenvalues of the Levi form is greater than or equal to
the trace of the Levi form with respect to the subbundle.  As shown in Theorem 1.9.9 in \cite{Zam08}, this implies that for $L^2$ $(0,q+1)$-forms $f$ in the kernel of $\dbar$, there exists an $L^2$ $(0,q)$-
form $u$ solving $\dbar u=f$.  To be consistent with the notation convention in $Z(q)$, we will typically refer to this as $(q-1)$-pseudoconvexity.

In this paper, we will generalize $(q-1)$-pseudoconvexity as follows.  Taking the trace of the Levi form with respect to a vector bundle can be thought of in local coordinates as taking the trace of the Levi
form with respect to a projection matrix (i.e., a hermitian matrix with eigenvalues of 0 or 1).  We will relax this condition by allowing eigenvalues that are between 0 and 1.  This gives us needed flexibility, as
demonstrated by an example in Proposition \ref{prop:key_counterexample} in which the rank of the identity minus our matrix is forced to be locally nonconstant.  By analogy with the nondegenerate case, we
will call such domains weak $Z(q)$ domains (see Definition \ref{defn:weak_Zq} for a formal definition).  The example in Proposition \ref{prop:key_counterexample} thus illustrates that weak $Z(q)$ is a
strictly weaker condition than any previously known condition.

We also allow for the boundary to be disconnected, using techniques developed for the annulus between two pseudoconvex domains in \cite{Shaw85a}.  These techniques allow us to solve $\dbar$ modulo
the space of harmonic $(0,q)$-forms in some weighted $L^2$ space.  By adapting recent arguments of Shaw \cite{Shaw10}, we are able to show that the space of harmonic $(0,q)$-forms in fact vanishes,
which allows us to use H\"ormander's methods to obtain results in unweighted $L^2$ spaces.  Our main $L^2$-result is thus the following:
\begin{thm}
\label{thm:main_theorem1}
  Let $M$ be an $n$-dimensional Stein manifold, and let $\Omega$ be a bounded subset of $M$ with $C^3$ boundary satisfying weak $Z(q)$ for some $1\leq q\leq n-1$.  Then we have
  \begin{enumerate}
    \item The space of harmonic $(0,q)$-forms $\mathcal{H}^q(\Omega)$ is trivial.

    \item The $\dbar$-Laplacian $\Box^q$ has closed range in $L^2_{(0,q)}(\Omega)$.

    \item The $\dbar$-Neumann operator $N^q$ exists and is continuous in $L^2_{(0,q)}(\Omega)$.

    \item The operator $\dbar$ has closed range in $L^2_{(0,q)}(\Omega)$ and $L^2_{(0,q+1)}(\Omega)$.

    \item The operator $\dbar^*$ has closed range in $L^2_{(0,q)}(\Omega)$ and $L^2_{(0,q-1)}(\Omega)$.
  \end{enumerate}
\end{thm}

We work in Stein manifolds motivated by results in \cite{HaLa75}.  The $C^3$ boundary is needed for our method of proof because of additional integration by parts that are carried out to handle the nonpositive eigenvalues of the Levi form.  Unfortunately, while the results of Theorem \ref{thm:main_theorem1} do not depend on the metric, our condition appears to depend on the metric (see Proposition \ref{prop:metric_counterexample}).  See \cite{Str11} for discussion of analogous difficulties surrounding the apparent metric dependence of Property $(P_q)$.  This is another benefit of working in generic Stein manifolds, since our example in Section \ref{sec:example} requires a non-Euclidean metric.  Our condition also implies stronger results which do depend on the metric.  Let $\varphi$ be a global plurisubharmonic exhaustion function for $M$, and let the metric for $M$ be given by the K\"ahler form $\omega=i\ddbar\varphi$.  For a weight function $\phi$ which is chosen to equal $\pm\varphi$ in a neighborhood of each connected component of $\partial\Omega$, we can define the weighted $L^2$-norm $\norm{f}^2_t=\int_\Omega e^{-t\phi}\abs{f}^2 dV$.  For sufficiently large $t>0$, we will be able to obtain Sobolev space estimates for the weighted $\dbar$-Neumann operator defined with respect to this norm.
\begin{thm}
\label{thm:main_theorem2}
  Let $M$ be an $n$-dimensional Stein manifold, and let $\Omega$ be a bounded subset of $M$ with $C^3$ boundary satisfying weak $Z(q)$ for some $1\leq q\leq n-1$.  Then there exists a constant $\tilde t>0$ such that for all $t>\tilde t$  and $-\frac{1}{2}\leq s\leq 1$ we have
  \begin{enumerate}
    \item The weighted $\dbar$-Neumann operator $N^q_t$ exists and is continuous in $L^2_{(0,q)}(\Omega)$.

    \item The canonical solution operators to $\dbar$ given by\\
    $\dbar^*_t N_t^q:W^s_{(0,q)}(\Omega)\rightarrow W^s_{(0,q-1)}(\Omega)$ and $N_t^q\dbar^*_t:W^s_{(0,q+1)}(\Omega)\rightarrow W^s_{(0,q)}(\Omega)$ are continuous.

    \item The canonical solution operators to $\dbar^*_t$ given by\\
    $\dbar N_t^q:W^s_{(0,q)}(\Omega)\rightarrow W^s_{(0,q+1)}(\Omega)$ and $N_t^q\dbar:W^s_{(0,q-1)}(\Omega)\rightarrow W^s_{(0,q)}(\Omega)$ are continuous.

    \item For every $f\in W^s_{(0,q)}(\Omega)\cap\ker\dbar$ there exists a $u\in W^s_{(0,q-1)}(\Omega)$ such that $\dbar u=f$.
  \end{enumerate}
\end{thm}

Interestingly, Theorem \ref{thm:main_theorem2} is actually needed to prove Theorem \ref{thm:main_theorem1} when the boundary is disconnected; see Section \ref{sec:proofs_main_theorems} for details.
The estimates of Theorem \ref{thm:main_theorem2} were carried out by the first author
in \cite{Har09} for pseudoconvex domains with $C^2$ boundary; again, the additional integration by parts needed for weak $Z(q)$ seem to require an additional degree of smoothness.  Furthermore, the elliptic regularization carried out in Section \ref{sec:Sobolev} seems to require a $C^3$ boundary.

When $\partial\Omega$ satisfies weak $Z(q)$ and weak $Z(n-q-1)$, we say that $\partial\Omega$ satisfies \emph{weak $Y(q)$}.
 In \cite{HaRa11}, the authors continued work of \cite{Nic06} and \cite{Rai10c} to understand solvability for the boundary operator $\dbar_b$ on $CR$-manifolds of hypersurface type.  The definition given for weak $Y(q)$ in that paper is completely superseded by the definition in the present paper.  When our bounded weak $Y(q)$ manifold is an actual hypersurface in a Stein manifold, we now have the following result:
\begin{thm}
\label{thm:main_theorem3}
  Let $M$ be an $n$-dimensional Stein manifold, and let $\Omega$ be a bounded subset of $M$ with connected $C^3$ boundary satisfying weak $Y(q)$ for some $1\leq q<n-1$.  For every $f\in L^2_{(0,q)}(\partial\Omega)\cap\ker\dbar_b$ there exists a $u\in L^2_{(0,q-1)}(\partial\Omega)$ satisfying $\dbar_b u=f$.  Hence, $\dbar_b$ has closed range in $L^2_{(0,q)}(\partial\Omega)$.
\end{thm}
Additional assumptions seem necessary for $q=n-1$, even in the pseudoconvex case (see \cite{ChSh01} for details).

\section{Basic Properties and Notation}

Let $M$ be an $n$-dimensional Stein manifold, $n\geq 2$, and fix a smooth, strictly plurisubharmonic exhaustion function $\varphi$ for $M$.  We endow $M$ with the K\"ahler metric given by the K\"ahler form $\omega=i\ddbar\varphi$.  In local coordinates $z_1,\ldots,z_n$, we will write
\[
  \omega=i \sum_{j,k=1}^n g_{j\bar k}dz_j\wedge d\bar{z}_k=i \sum_{j,k=1}^n \frac{\partial^2\varphi}{\partial z_j\partial\bar z_k}dz_j\wedge d\bar{z}_k.
\]
As usual, $g^{\bar k j}$ will denote the inverse matrix to $g_{j\bar k}$.  By the usual convention, we will use the metric to raise and lower indices, so that, for example,
\[
  \sum_{\ell=1}^n g_{j\bar\ell}b^{\bar\ell k}=b_j^{\cdot k}\text{ and }
  \sum_{\ell=1}^n c_{j\bar\ell}g^{\bar\ell k}=c_j^{\cdot k}.
\]
Additionally, we use the bracket $\langle\cdot,\cdot\rangle$ notation for the metric pairing, i.e., if $L,L'\in T^{1,0}(M)$, then
$\langle L, L' \rangle = \omega(i\bar L'\wedge L)$ whereas if $\alpha = \sum_{j=1}^n a_j\, dz_j$ and $\alpha' = \sum_{j=1}^n a_j' \, dz_j$ in local coordinates, then
$\langle \alpha,\alpha'\rangle = \sum_{j,k=1}^n\bar{a}_k' g^{\bar k j} a_j $. By multilinearity, $\langle \cdot,\cdot\rangle$ now extends to $(p,q)$-forms.

Let $\Omega\subset M$ be a bounded domain with $C^m$ boundary.  By definition, this means that there exists a $C^m$ function $\rho$ on $M$ such that $\Omega=\set{z\in M:\rho(z)<0}$ and $d\rho\neq 0$ on $\partial\Omega$.  Such a $\rho$ is called a $C^m$ defining function for $\Omega$.  For $z\in\partial\Omega$, we define the induced CR-structure on $\partial\Omega$ at $z$ by
\[
  T^{1,0}_z(\partial\Omega)=\set{L\in T^{1,0}_z(M):\partial\rho(L)=0}.
\]
Let $T^{1,0}(\partial\Omega)$ denote the space of $C^{m-1}$ sections of $T^{1,0}_z(\partial\Omega)$.  We will also need $T^{0,1}(\partial\Omega)=\overline{T^{1,0}(\partial\Omega)}$ and the exterior algebra generated by these: $T^{p,q}(\partial\Omega)$.  Let $\Lambda^{p,q}(\partial\Omega)$ denote the bundle of $C^{m-1}$ $(p,q)$-forms on $T^{p,q}(\partial\Omega)$.  We use $\tau$ to denote the orthogonal projection and restriction:
\begin{equation}\label{eqn:tau def'n}
  \tau:\Lambda^{p,q}(M)\rightarrow\Lambda^{p,q}(\partial\Omega).
\end{equation}
For each element $X$ of $T^{p,q}$ (resp. $\Lambda^{p,q}$), we denote the metric dual element of $\Lambda^{p,q}$ (resp. $T^{p,q}$) by $X^\sharp$.  This is defined to satisfy the relationships $X^\sharp(Y)=\left<Y,X\right>$ for all $Y\in T^{p,q}$ (resp. $Y(X^\sharp)=\left<Y,X\right>$ for all $Y\in \Lambda^{p,q}$).  For example, the dual of the K\"ahler form is given in local coordinates by
\[
  \omega^\sharp=i\sum_{j,k=1}^n g^{\bar k j}\frac{\partial}{\partial\bar z_k}\wedge\frac{\partial}{\partial z_j}.
\]

For any $C^2$ defining function $\rho$, the \emph{Levi form $\mathcal{L}_\rho$} is the real element of $\Lambda^{1,1}(\partial\Omega)$ defined by
\[
  \mathcal{L}_\rho(i \bar L\wedge L')=i\ddbar\rho(i \bar L\wedge L')
\]
for any $L, L'\in T^{1,0}(\partial\Omega)$.  As usual, if $\tilde\rho$ is another $C^2$ defining function for $\Omega$, then $\tilde\rho=\rho h$ for some nonvanishing $C^1$ function $h$, and $\mathcal{L}_{\tilde\rho}=h\mathcal{L}_\rho$.  We will typically suppress the subscript $\rho$ when the choice of defining function is not relevant.
\begin{defn}
\label{defn:weak_Zq}
  For $1\leq q\leq n-1$, we say $\partial\Omega$ satisfies \emph{$Z(q)$ weakly} if there exists a real $\Upsilon\in T^{1,1}(\partial\Omega)$ satisfying
  \begin{enumerate}
    \item \label{item:upper_and_lower_bounds} $\abs{\theta}^2\geq (i\theta\wedge\bar\theta)(\Upsilon)\geq 0$ for all $\theta\in \Lambda^{1,0}(\partial\Omega)$.

    \item \label{item:Levi form_property} $\mu_1+\cdots+\mu_q-\mathcal{L}(\Upsilon)\geq 0$ where $\mu_1,\ldots,\mu_{n-1}$ are the eigenvalues of $\mathcal{L}$ in increasing order.

    \item \label{item:trace} $\omega(\Upsilon)\neq q$.
  \end{enumerate}
\end{defn}
\begin{rem}
  Note that this is an intrinsic definition, so it can also be applied to abstract CR-manifolds of hypersurface type.
  This replaces the definition given in \cite{HaRa11}, and the main results of that paper still follow with this more general definition.
\end{rem}

The fact that $\partial\Omega$ is the boundary of a domain induces a natural orientation on $\partial\Omega$.  It is sometimes useful to reverse the orientation and think of $\partial\Omega$ as the boundary of the complement instead.  The following observation is trivial for $Z(q)$, and motivates the definition of $Y(q)$, so it is of interest to confirm the corresponding fact for weak $Z(q)$.
\begin{prop}
\label{prop:Yq_motivation}
  For $1\leq q\leq n-2$, let $\Omega\subset M$ be a bounded domain and let $B\subset M$ be a sufficiently large bounded pseudoconvex domain so that $\overline\Omega\subset B$.  Then $\partial\Omega$ satisfies $Z(q)$ weakly if and only if $\partial(B/\overline\Omega)$ satisfies $Z(n-q-1)$ weakly.
\end{prop}
\begin{proof}
  Suppose $\partial\Omega$ satisfies $Z(q)$ weakly, and let $\tilde{\Upsilon}$ be the element of $T^{1,1}(\partial\Omega)$ given by Definition \ref{defn:weak_Zq}.  On $\partial B$, we can define $\Upsilon=0$, and on $\partial\Omega$ we define $\Upsilon=(\tau\omega)^\sharp-\tilde\Upsilon$.  If $\tilde\mu_1,\ldots,\tilde\mu_{n-1}$ are the eigenvalues of the Levi form of $\partial\Omega$ in increasing order, then $\mu_1=-\tilde\mu_{n-1},\ldots,\mu_{n-1}=-\tilde\mu_1$ are the eigenvalues of the Levi form of $\partial(B/\overline\Omega)$ (on $\partial\Omega$) in increasing order, so since $\mathcal{L}((\tau\omega)^\sharp)=\mu_1+\cdots+\mu_{n-1}$, we have
  \[
    \mu_1+\cdots+\mu_{n-q-1}-\mathcal{L}(\Upsilon)=\tilde\mu_1+\cdots+\tilde\mu_q-\tilde{\mathcal{L}}(\tilde\Upsilon).
  \]
  Furthermore, $\omega(\Upsilon)=n-1-\omega(\tilde\Upsilon)\neq n-q-1$.

  Similar computations prove the converse.
\end{proof}
\begin{rem}
\label{rem:Yq_motivation}
  We can replace $B$ with any bounded domain such that $\partial B$ satisfies $Z(n-q-1)$ weakly.
\end{rem}
Motivated by this, we define
\begin{defn}
  For $1\leq q\leq n-2$, we say $\partial\Omega$ satisfies \emph{$Y(q)$ weakly} if $\partial\Omega$ satisfies $Z(q)$ weakly and $Z(n-q-1)$ weakly.
\end{defn}

We note that Definition \ref{defn:weak_Zq} is essentially a local property (modulo connected boundary components).
\begin{lem}
\label{lem:local}
  For $1\leq q\leq n-1$, let $\sigma:\partial\Omega\rightarrow\set{-1,1}$ be continuous and suppose that for every $p\in\partial\Omega$ there exists an open neighborhood $U_p$ of $p$ such that $U_p\cap\partial\Omega$ is connected and a real $\Upsilon_p\in T^{1,1}(U_p)$ satisfying
  \begin{enumerate}
    \item $\abs{\theta}^2\geq (i\theta\wedge\bar\theta)(\Upsilon_p)\geq 0$ for all $\theta\in \Lambda^{1,0}(U_p)$.

    \item $\mu_1+\cdots+\mu_q-\mathcal{L}(\Upsilon_p)\geq 0$ on $U_p$ where $\mu_1,\ldots,\mu_{n-1}$ are the eigenvalues of the Levi form in increasing order.

    \item $\sigma(p)\left(\omega(\Upsilon_p)-q\right)>0$ on $U_p$.
  \end{enumerate}
  Then $\partial\Omega$ satisfies $Z(q)$ weakly.
\end{lem}
\begin{rem}
  The function $\sigma$ represents a choice of orientation for each connected boundary component of $\Omega$.
\end{rem}
\begin{proof}
  Choose a finite cover $\set{U_p}_{p\in\mathcal{P}}$ of $\partial\Omega$ and let $\chi_p$ be a subordinate partition of unity.  If we let $\Upsilon=\sum_{p\in\mathcal{P}}\chi_p\Upsilon_p$, then the necessary properties are satisfied by linearity.  Since $\sigma$ is constant on each connected component of $\partial\Omega$, $\omega(\Upsilon_p)-q$ will have a constant sign on each connected component of $\partial\Omega$, so there is no possibility of cancellation in the corresponding sum.
\end{proof}

Fix $p\in\partial\Omega$, and choose local coordinates that are orthonormal at $p$ and satisfy $\frac{\partial\rho}{\partial z_j}(p)=0$ for all $1\leq j\leq n-1$.  At $p$ we can write
\[
  \Upsilon=i\sum_{j,k=1}^{n-1} b^{\bar k j}\frac{\partial}{\partial\bar z_k}\wedge\frac{\partial}{\partial z_j}\quad\text{ and }\quad
  \mathcal{L}=i\sum_{j,k=1}^{n-1} c_{j\bar k}dz_j\wedge d\bar z_k
\]
where $b^{\bar k j}$ and $c_{j\bar k}$ are hermitian $(n-1)\times(n-1)$ matrices.  Suppose that the local coordinates are chosen to diagonalize $b^{\bar k j}$ at $p$, and write $b^{\bar k j}=\delta_{jk}\lambda_j$.  When restricted to $p$, the defining characteristics of weak $Z(q)$ will take the form
\begin{enumerate}
  \item $0\leq\lambda_j\leq 1$ for all $1\leq j\leq n-1$.

  \item $\mu_1+\cdots+\mu_q-(\lambda_1 c_{1\bar 1}+\cdots+\lambda_{n-1}c_{n-1\overline{n-1}})\geq 0$.

  \item $\lambda_1+\cdots+\lambda_{n-1}\neq q$.
\end{enumerate}
If there is an orthonormal local coordinate frame that diagonalizes $b^{\bar k j}$ such that $\lambda_1=\cdots=\lambda_m=1$ and $\lambda_{m+1}=\cdots=\lambda_{n-1}=0$ for some $m\neq q$, then this is the condition studied in \cite{HaRa11} which was shown to generalize $Z(q)$ and $(q-1)$-pseudoconvexity \cite{Zam08} (with the weight $\varphi(z)=\abs{z}^2$ in \cite{HaRa11}).

Alternatively, we can choose orthonormal coordinates that diagonalize the Levi form at a point, so that $c_{j\bar k}=\delta_{jk}\mu_j$.  The second condition then translates into
\[
  \mu_1(1-b^{\bar 1 1})+\cdots+\mu_q(1-b^{\bar q q})-(\mu_{q+1} b^{\overline{q+1}q+1}+\cdots+\mu_{n-1} b^{\overline{n-1}n-1})\geq 0.
\]
Since $(1-b^{\bar j j})\geq 0$, $b^{\bar j j}\geq 0$, and $\mu_j\leq \mu_q\leq \mu_{q+1}\leq \mu_k$ for all $j\leq q< q+1\leq k$, it follows that
$\mu_q(q-\omega(\Upsilon))\geq 0$ and $\mu_{q+1}(q-\omega(\Upsilon))\geq 0$.  Hence, if $\mu_q<0$,
then $\omega(\Upsilon)>q$, and if $\mu_{q+1}>0$, then $\omega(\Upsilon)<q$.  Equivalently, we have
\begin{lem}
  \label{lem:count_eigenvalues}
  For $1\leq q\leq n-1$ let $\Omega\subset M$ be a bounded domain and suppose that $\partial\Omega$ satisfies $Z(q)$ weakly.  Let $\Upsilon$ be as in Definition \ref{defn:weak_Zq}.  For any fixed boundary point, if $\omega(\Upsilon)<q$ then the Levi form has at least $n-q$ nonnegative eigenvalues, and if $\omega(\Upsilon)>q$, then the Levi form has at least $q+1$ nonpositive eigenvalues.
\end{lem}

Many results will be easier to work with when the boundary is connected.  The fact that we are only working with bounded domains induces a natural decomposition into domains with connected boundaries.
\begin{lem}
\label{lem:connected_decomposition}
  For any Stein manifold $M$, supposed that $\Omega\subset M$ is a connected bounded domain with $C^3$ boundary satisfying $Z(q)$ weakly for some $1\leq q\leq n-1$.  Then $\Omega=\Omega_1/\bigcup_{j=2}^m\overline\Omega_j$ where $\Omega_j$ has connected boundary for each $1\leq j\leq m$, $\Omega_1$ satisfies $Z(q)$ weakly, and $\Omega_j$ satisfies $Z(n-q-1)$ weakly for each $2\leq j\leq m$.  
The $(1,1)$-vector $\Upsilon$ in definition \ref{defn:weak_Zq} 
will satisfy $\omega(\Upsilon)<q$ on $\partial\Omega_1$ and $\omega(\Upsilon)>q$ on $\partial\Omega_j$ for $2\leq j\leq m$.
\end{lem}
\begin{proof}
Let $\psi:M\rightarrow\mathbb{C}^{2n+1}$ be an embedding (see Theorem 5.3.9 in \cite{Hor90}).
Since $\Omega$ is bounded, there exists a minimal radius $R>0$ such that $\psi[\Omega]$ is contained in a ball centered at zero with radius $R$.  Denote the pullback of this ball under $\psi$ by $B$.
Then $B$ is a strictly pseudoconvex domain in $M$ containing $\Omega$ and since $R$ is minimal there exists at least one point $p\in\partial\Omega\cap\partial B$.  At $p$, $\partial\Omega$ must also be
strictly pseudoconvex, so $\omega(\Upsilon)(p)<q$ by
the contrapositive of Lemma \ref{lem:count_eigenvalues}.  By continuity, $\omega(\Upsilon)<q$ on the connected boundary component containing $p$, so we define this to be $\partial\Omega_1$.
Since $\Omega$ is connected, the remaining boundary components (finitely many, since $\Omega$ is relatively compact with $C^3$ boundary) can be thought of as boundaries of $Z(n-q-1)$ domains by
Proposition \ref{prop:Yq_motivation} and Remark \ref{rem:Yq_motivation}.  Using the argument with a ball in $\mathbb{C}^{2n+1}$, we again see that each of these subdomains has a strictly pseudoconvex
point, and hence the Levi form is negative definite (when viewed as part of $\partial\Omega$).  When the Levi form is negative-definite, we must have $\omega(\Upsilon)>q$ by Lemma
\ref{lem:count_eigenvalues}.
\end{proof}
\begin{rem}
  \label{rem:Z(n-1)} One consequence of this proof is that there are no bounded weak $Z(0)$ domains, since $\omega(\Upsilon)<0$ is impossible.  On the other hand, bounded weak $Z(n-1)$ domains can exist (e.g., pseudoconvex domains), but they must have connected boundaries (otherwise some boundary components would bound weak $Z(0)$ domains).  For analysis of the $q=n-1$ case on domains with disconnected boundaries, see \cite{Hor04} and \cite{Shaw10}.
\end{rem}

To prove our basic estimates, we will need extensions of $\Upsilon$ to $M$.
\begin{lem}
\label{lem:upsilon_extension}
  Suppose that $\partial\Omega$ satisfies $Z(q)$ weakly, and let $\Upsilon$ be as in Definition \ref{defn:weak_Zq}.  Let $\rho$ be any $C^m$ defining function for $\Omega$.  There exist relatively compact open sets $U^+$, $U^-$, and $U^0$ covering $\overline\Omega$ such that $\partial\Omega\cap\overline U^0=\emptyset$ and $\overline U^+\cap\overline U^-=\emptyset$, along with real $\Upsilon^+,\Upsilon^-\in T^{1,1}(M)$ satisfying
  \begin{enumerate}
    \item \label{item:positivity} $\abs{\theta}^2\geq (i\theta\wedge\bar\theta)(\Upsilon^\pm)\geq 0$ for all $\theta\in \Lambda^{1,0}(M)$.

    \item \label{item:nondegeneracy} $\omega(\Upsilon^+)<q$ and $\omega(\Upsilon^-)>q$ on $M$.

    \item \label{item:projection} For any $\theta\in\Lambda^{1,0}(M)$ we have
    \[
      (i\theta\wedge\bar\theta)(\Upsilon^\pm)=(i\tau\theta\wedge\tau\bar\theta)(\Upsilon)
    \]
    on $\partial\Omega\cap U^\pm$.

    \item On a neighborhood of $\partial\Omega\cap U^\pm$, we have $(\theta\wedge\dbar\rho)(\Upsilon^\pm)=0$ for all $\theta\in \Lambda^{1,0}(M)$.
  \end{enumerate}
\end{lem}
\begin{rem}
\label{rem:connected_boundary}
  Note that $U^+\neq\emptyset$ by Lemma \ref{lem:connected_decomposition}.  On the other hand, if $\partial\Omega$ is connected, we can set $U^-=\emptyset$ and $U^0=\emptyset$.
\end{rem}
\begin{proof}
Let $K^+=\set{z\in\partial\Omega:\omega(\Upsilon)<q}$ and $K^-=\set{z\in\partial\Omega:\omega(\Upsilon)>q}$.  By the continuity of $\omega(\Upsilon)$ these are disconnected from each other, so there
exist open neighborhoods $U^+$ and $U^-$ such that $K^\pm\subset U^\pm$ and $K^\pm\cap\overline U^\mp=\emptyset$.  Choose $U^0$ such that
$\partial\Omega\cap\overline U^0=\emptyset$ and $\overline\Omega\subset U^0\cup U^+\cup U^-$.

Let $U$ be a neighborhood of $\partial\Omega$ on which $d\rho\neq 0$ for some $C^m$ defining function $\rho$.  Let $\psi(w,t):U\times[0,\abs{\rho(w)}]\rightarrow U$ solve the initial value problem
$\psi(w,0)=w$ and $\frac{\partial}{\partial t}\psi(w,t)=-\left((\mathrm{sgn}\rho)\abs{d\rho}^{-2}(d\rho)^\sharp\right)(\psi(w,t))$.  By construction, this will satisfy
\[
\frac{\p}{\p t}\rho(\psi(w,t)) = -(\mathrm{sgn}\rho)|d\rho|^{-2}\langle d\rho,d\rho\rangle(\psi(w,t)) = -(\mathrm{sgn}\rho)(\psi(w,t)),
\]
so $\psi(w,\abs{\rho(w)})\in\partial\Omega$.

We denote parallel translation along $\psi(w,t)$ by
  \[
    P_{a,w}^b:T^{p,q}_{\psi(w,a)}(M)\rightarrow T^{p,q}_{\psi(w,b)}(M).
  \]

  Choose $\chi\in C^\infty_0(U)$ such that $\chi\equiv 1$ on a neighborhood of $\partial\Omega$.  We define $\Upsilon^\pm$ on $\partial\Omega\cap U^\pm$ by
  \[
    (i\theta\wedge\bar\theta)(\Upsilon^\pm)=(i\tau\theta\wedge\tau\bar\theta)(\Upsilon)
  \]
  for any $\theta\in\Lambda^{1,0}(M)$.  On $U\cap U^\pm$, we parallel translate $\Upsilon^\pm$ along $\psi$, as follows.  We define
  \[
    \Upsilon^+(w)=\chi(w)P_{\abs{\rho(w)},w}^{0}\Upsilon^+(\psi(w,\abs{\rho(w)}))
  \]
  and
  \[
    \Upsilon^-(w)=\chi(w)P_{\abs{\rho(w)},w}^{0}\Upsilon^-(\psi(w,\abs{\rho(w)}))+(1-\chi(w))\omega^\sharp.
  \]
  We can now define $\Upsilon^+=0$ on $M/(U\cap U^+)$ and $\Upsilon^-=\omega^\sharp$ on $M/(U\cap U^-)$.
\end{proof}

We are now ready to define our weight function.  Let $U^\pm$, $U^0$, and $\Upsilon^\pm$ be as in Lemma \ref{lem:upsilon_extension}.  Fix $\chi\in C^\infty_0(M/\overline U^-)$ such that $\chi\equiv 1$ on $\overline U^+$.  Set
\[
  \phi=\chi\varphi-(1-\chi)\varphi.
\]
While the complex Hessian of $\phi$ on $U^0$ will involve derivatives of $\chi$, we still have
\begin{equation}
\label{eq:complex_hessian}
  i\ddbar\phi=\begin{cases}\omega &\text{on }U^+\\-\omega &\text{on }U^-\end{cases}.
\end{equation}
We next define the usual weighted $L^2$-inner products.  For $f, h\in L^2_{(0,q)}(\Omega)$, define
\[
  (f,h)_t=\int_\Omega e^{-t\phi}\left<f,h\right>dV.
\]
and $\norm{f}_t=\sqrt{(f,f)_t}$.  Since $e^{-t\phi}$ is uniformly bounded on $\Omega$, the space of $(0,q)$-forms bounded in $\norm{\cdot}_t$ is equal to $L^2_{(0,q)}(\Omega)$.  The operator
\[
  \dbar:L^2_{(0,q)}(\Omega,e^{-t\phi})\rightarrow L^2_{(0,q+1)}(\Omega,e^{-t\phi})
\]
is given its $L^2$-maximal definition, and the adjoint
\[
  \dbar^*_t:L^2_{(0,q+1)}(\Omega,e^{-t\phi})\rightarrow L^2_{(0,q)}(\Omega,e^{-t\phi})
\]
is defined with respect to the weighted inner product $(\cdot,\cdot)_t$.  We also have $\Box_t^q=\dbar\dbar^*_t+\dbar^*_t\dbar$ with the induced domain.  The space of harmonic forms is given by
\[
  \mathcal{H}^q_t(\Omega)=L^2_{(0,q)}(\Omega,e^{-t\phi})\cap\ker\dbar\cap\ker\dbar^*_t,
\]
with the projection onto these denoted $H_t^q:L^2_{(0,q)}(\Omega,e^{-t\phi})\rightarrow\mathcal{H}_t^q(\Omega,e^{-t\phi})$.  When it exists, the weighted $\dbar$-Neumann operator
\[
  N^q_t:L^2_{(0,q)}(\Omega,e^{-t\phi})\rightarrow\dom{\Box_t^q}
\]
satisfies $\Box_t^q N_t^q=I-H_t^q$.

Let $\mathcal{I}_q$ denote the set of increasing multi-indices over $\set{1,\cdots,n}$ of length $q$.  For an open set $U\subset M$ with local coordinates $\set{z_1^U,\ldots, z_n^U}$, we let $\nabla_j^U$ denote the covariant derivative with respect to $\frac{\partial}{\partial z_j^U}$.  We also use $\nabla_{j,t}^{U,*}=-\overline\nabla_j^U+t\frac{\partial\phi}{\partial\bar z_j^U}$.  This satisfies the adjoint relationship
\[
  \sum_{j,k=1}^n(g^{\bar k j}_U\nabla_j^U f,h_k)_t=\sum_{j,k=1}^n(g^{\bar k j}_U f,\nabla_{j,t}^{U,*}h_k)_t
\]
assuming $f$ and $h_k$ are compactly supported.  If $\mathcal{U}$ is a finite open cover of $\Omega$ by such sets, we let $\set{\chi^U}_{U\in\mathcal{U}}$ denote a partition of unity subordinate to this cover and define the following gradient terms on $(0,q)$-forms $f$ for any $\Upsilon\in T^{1,1}(M)$:
\begin{align}
\label{eq:bar_gradient}
  \norm{\overline\nabla f}^2_t&=\sum_{U\in\mathcal{U}}\sum_{j,k=1}^n(\chi^U g^{\bar k j}_U\overline\nabla^U_k f,\overline\nabla^U_j f)_t\\
\label{eq:bar_gradient_upsilon}
  \norm{\overline\nabla_\Upsilon f}^2_t&=\sum_{U\in\mathcal{U}}\sum_{j,k=1}^n(\chi^U b^{\bar k j}_U\overline\nabla^U_k f,\overline\nabla^U_j f)_t\\
\label{eq:gradient_upsilon}
  \norm{\nabla_\Upsilon f}^2_t&=\sum_{U\in\mathcal{U}}\sum_{j,k=1}^n(\chi^U b^{\bar k j}_U\overline\nabla^{U,*}_{j,t} f,\overline\nabla^{U,*}_{k,t} f)_t,
\end{align}
where
\[
  \Upsilon=i\sum_{j,k=1}^n b^{\bar k j}_U \frac{\partial}{\partial\bar z_k^U}\wedge\frac{\partial}{\partial z_j^U}
\]
on $U$.  We also introduce vector fields which will figure prominently in our error terms:
\begin{align*}
  E&=\sum_{U\in\mathcal{U}}\sum_{j,k,\ell}\chi^U g^{\bar k\ell}_U\left(\frac{\partial}{\partial\bar z_k^U}b_{\ell, U}^{\cdot j}\right)\frac{\partial}{\partial z_j^U}\\
  E_\Upsilon&=\sum_{U\in\mathcal{U}}\sum_{j,k,\ell,r}\chi^U g^{\bar k\ell}_U\left(\frac{\partial}{\partial\bar z_k^U}b_{\ell, U}^{\cdot r}\right)b_{r, U}^{\cdot j}\frac{\partial}{\partial z_j^U}.
\end{align*}
Note that if we change coordinates, $b_{\ell, U}^{\cdot r}$ will be multiplied by matrices of holomorphic functions, which will be annihilated by $\frac{\partial}{\partial\bar z_k^U}$,
so the vector fields remain invariant under changes of coordinates.
At any point $p\in\Omega$, choose orthonormal coordinates at $p$ that diagonalize $b^{\bar k j}$, with eigenvalues $\lambda_j$ corresponding to the eigenvector $\frac{\partial}{\partial z_j}$ at $p$.  If $\Upsilon$ satisfies property \eqref{item:positivity} in Lemma \ref{lem:upsilon_extension}, then $0\leq\lambda_j\leq 1$.  If $\Upsilon$ is $C^1$, then at $p$ we can write
\[
  E=\sum_{j=1}^n A^j\frac{\partial}{\partial z_j}\text{ and }E_\Upsilon=\sum_{j=1}^n A^j\lambda_j\frac{\partial}{\partial z_j}
\]
where $A^j$ are continuous functions on our local coordinate patch.  Hence, at $p$, since $\lambda_j^2\leq\lambda_j$ we have
\[
  \abs{\overline\nabla^*_{E_\Upsilon,t} f}^2=\abs{\sum_{j=1}^n A^j\lambda_j\overline\nabla^*_{j,t}f}^2\leq C\sum_{j=1}^n\lambda_j\abs{\overline\nabla^*_{j,t}f}^2
\]
for some constant $C>0$.  Integrating, this gives us
\begin{equation}
\label{eq:gradient_error_estimate}
  \norm{\overline\nabla^*_{E_\Upsilon,t} f}^2_t\leq C\norm{\nabla_\Upsilon f}^2_t.
\end{equation}
On the other hand, since $(1-\lambda_j)^2\leq(1-\lambda_j)$, we also have at $p$
\[
  \abs{\overline\nabla_E f-\overline\nabla_{E_\Upsilon} f}^2=\abs{\sum_{j=1}^n A_j(1-\lambda_j)\overline\nabla_j f}^2\leq C\sum_{j=1}^n(1-\lambda_j)\abs{\overline\nabla_j f}^2
\]
for some constant $C>0$.  Integration gives us
\begin{equation}
\label{eq:bar_gradient_error_estimate}
  \norm{\overline\nabla_E f-\overline\nabla_{E_\Upsilon} f}^2_t\leq C\left(\norm{\overline\nabla f}^2_t-\norm{\overline\nabla_\Upsilon f}^2_t\right).
\end{equation}

We also abuse notation and define the action of $(1,1)$-forms on $(0,q)$-forms.  Let $f\in C^1_{(0,q)}(\overline\Omega)\cap\dom{\dbar^*_t}$.  For any point $p\in\Omega$, choose local coordinates that are orthonormal at $p$ and define
\[
  i\ddbar\phi(f,f)(p)=\sum_{J\in\mathcal{I}_{q-1}}\sum_{j,k=1}^n\frac{\partial^2\phi}{\partial z_j\partial\bar z_k}f_{kJ}\bar f_{jJ}
\]
where $f_{kJ} = (-1)^\sigma f_K$ for $K\in\mathcal I_q$ if $\{k\}\cup J = K$ as sets and $\sigma$ is the length of the permutation that changes $kJ$ into $K$.  Due to \eqref{eq:complex_hessian}, we have
\begin{equation}
\label{eq:hessian_bound}
  i\ddbar\phi(f,f)=\begin{cases}q\abs{f}^2 &\text{on }U^+\\-q\abs{f}^2 & \text{on }U^-\end{cases}.
\end{equation}
For any point $p\in\partial\Omega$, choose local coordinates that are orthonormal at $p$ such that $\frac{\partial\rho}{\partial z_j}(p)=0$ for all $1\leq j\leq n-1$.  We define
\[
  \mathcal{L}(f,f)(p)=\sum_{J\in\mathcal{I}_{q-1}}\sum_{j,k=1}^{n-1}\frac{\partial^2\rho}{\partial z_j\partial\bar z_k}f_{kJ}\bar f_{jJ}.
\]
We note for future reference that if $\mu_1,\ldots,\mu_{n-1}$ are the eigenvalues of $\mathcal{L}$ arranged in increasing order, then (adapting the proof of Lemma 4.7 in \cite{Str10}), we have
\begin{equation}
\label{eq:levi_bound}
  \mathcal{L}(f,f)\geq (\mu_1+\cdots+\mu_q)\abs{f}^2.
\end{equation}

\section{The Basic Estimate}
\label{sec:basic_estimate}

Let $\rho$ be a $C^m$ defining function for $\Omega$ with $\abs{d\rho}=1$ on $\Omega$.  For $f\in C^1_{(0,q)}(\overline{\Omega})\cap\dom{\dbar^*_t}$ with $1\leq q\leq n$, we have the Morrey-Kohn-H\"ormander equality (see for example \cite{ChSh01}, \cite{FoKo72}, \cite{Hor90}, or \cite{Str10}):
\begin{equation}
\label{eq:MKH}
  \norm{\dbar f}^2_t+\norm{\dbar^*_t f}^2_t=\norm{\overline\nabla f}^2_t+t \int_\Omega i\ddbar\phi(f,f)e^{-t\phi}dV+\int_{\partial\Omega}\mathcal{L}(f,f)e^{-t\phi}dS+O(\norm{f}^2_t).
\end{equation}
The error term involves the curvature of the K\"ahler metric, and can be computed explicitly using the Bochner-Kodaira technique \cite{Siu82}.  Since this term can be controlled by choosing $t$ large enough, we won't need the precise value.

We wish to understand integration by parts in the gradient term.  Note that in \eqref{eq:bar_gradient}, \eqref{eq:bar_gradient_upsilon}, and \eqref{eq:gradient_upsilon}, the integrated terms are invariant under changes of coordinate, so derivatives of the partition of unity $\chi^U$ that arise from integration by parts will cancel. Hence, for clarity of notation, we can suppress the partition of unity without losing information.

Let $U^\pm$, $U^0$, and $\Upsilon^\pm$ be as in Lemma \ref{lem:upsilon_extension}.  Note that property \eqref{item:projection} in this lemma guarantees that $\Upsilon^\pm$ has no normal component on $\partial\Omega\cap U^\pm$, so $\Upsilon^\pm|_{\partial\Omega\cap U^\pm}$ is made up of tangential derivatives which can be integrated by parts without introducing a boundary term.  Hence, working in local coordinates with $f\in C^2_{(0,q)}(\overline\Omega)\cap C^2_0(U^\pm)$ we have
\begin{equation}
\label{eq:first_int_by_parts}
\begin{split}
  \norm{\overline\nabla_{\Upsilon^\pm} f}^2_t&=\sum_{j,k=1}^n(b^{\bar k j}\overline\nabla_k f,\overline\nabla_j f)_t\\
  &=\sum_{j,k,\ell=1}^n(g^{\bar k\ell}f,\overline\nabla_{k,t}^* (b_{\bar\ell}^{\cdot\bar j}\overline\nabla_j f))_t\\
  &=\sum_{j,k=1}^n(b^{\bar k j}f,\overline\nabla_{k,t}^*\overline\nabla_j f)_t-(f,\overline\nabla_E f)_t.
\end{split}
\end{equation}
To continue, we will need the commutator
\[
  \sum_{j,k=1}^n b^{k\bar j}[\overline\nabla_{k,t}^*,\overline\nabla_j] f=-t\sum_{j,k=1}^n b^{k\bar j}\frac{\partial^2\phi}{\partial z_k\partial\bar z_j}f+O(f)=\mp t \omega(\Upsilon^\pm)f+O(f)
\]
where the error terms are independent of $t$ and \eqref{eq:complex_hessian} has been used.  Substituting in \eqref{eq:first_int_by_parts} we have
\begin{equation}
\label{eq:second_int_by_parts}
\begin{split}
  \norm{\overline\nabla_{\Upsilon^\pm} f}^2_t&=\mp t(\omega(\Upsilon^\pm)f,f)_t+\sum_{j,k=1}^n(b^{\bar k j}f,\overline\nabla_j\overline\nabla_{k,t}^* f)_t-(f,\overline\nabla_E f)_t+O(\norm{f}^2_t)\\
  &=\mp t(\omega(\Upsilon^\pm)f,f)_t+\sum_{j,k,\ell=1}^n\big (g^{\bar\ell j}\overline\nabla_{j,t}^* \big(b^{\bar k}_{\cdot\bar\ell}f\big),\overline\nabla_{k,t}^* f\big)_t-(f,\overline\nabla_E f)_t\\
  &\qquad+O(\norm{f}^2_t)\\
  &=\mp t(\omega(\Upsilon^\pm)f,f)_t+\norm{\nabla_{\Upsilon^\pm} f}^2_t-(f,\overline\nabla_{E,t}^* f)_t-(f,\overline\nabla_E f)_t+O(\norm{f}^2_t).
\end{split}
\end{equation}
When we integrate the error terms by parts, it will be helpful to note that on $\partial\Omega$ we have $E_{\Upsilon^\pm}\rho=0$ but
\[
  E\rho=\sum_{j,k,\ell}g^{\bar k\ell}\left(\frac{\partial}{\partial\bar z_k}b_{\ell}^{\cdot j}\right)\frac{\partial\rho}{\partial z_j}=-\sum_{j,k}b^{\bar k j}\frac{\partial^2 \rho}{\partial z_j\partial\bar z_k}=-\mathcal{L}(\Upsilon).
\]
It is also important to note that $\Upsilon^\pm$ must be $C^2$ if integration by parts with respect to $\nabla_{E_{\Upsilon^\pm}}$ is going to be well-defined.  Considering the error terms in \eqref{eq:second_int_by_parts}, we have
\begin{equation}
\label{eq:first_error_int_by_parts}
  (f,\overline\nabla_{E,t}^* f)_t=(f,\overline\nabla_{E_{\Upsilon^\pm},t}^* f)_t+((\overline\nabla_E-\overline\nabla_{E_{\Upsilon^\pm}})f,f)_t+\int_{\partial\Omega}\mathcal{L}(\Upsilon)\abs{f}^2 e^{-t\phi} dS+O(\norm{f}^2_t)
\end{equation}
and
\begin{equation}
\label{eq:second_error_int_by_parts}
  (f,\overline\nabla_E f)_t=(\overline\nabla_{E_{\Upsilon^\pm},t}^* f,f)_t+(f,(\overline\nabla_E-\overline\nabla_{E_{\Upsilon^\pm}})f)_t+O(\norm{f}^2_t).
\end{equation}
Substituting \eqref{eq:first_error_int_by_parts} and \eqref{eq:second_error_int_by_parts} into \eqref{eq:second_int_by_parts}, we have
\begin{equation}
\label{eq:int_by_part_identity}
\begin{split}
  \norm{\overline\nabla_{\Upsilon^\pm} f}^2_t&=\mp t(\omega(\Upsilon^\pm)f,f)_t-2\re\left((\overline\nabla_{E_{\Upsilon^\pm},t}^* f,f)_t+(f,(\overline\nabla_E-\overline\nabla_{E_{\Upsilon^\pm}})f)_t\right)\\
  &\qquad+\norm{\nabla_{\Upsilon^\pm} f}^2_t-\int_{\partial\Omega}\mathcal{L}(\Upsilon)\abs{f}^2 e^{-t\phi} dS+O(\norm{f}^2_t).
\end{split}
\end{equation}
Since we have property \eqref{item:positivity} in Lemma \ref{lem:upsilon_extension}, we can now write
\[
  \norm{\overline\nabla f}^2_t=\left(\norm{\overline\nabla f}^2_t-\norm{\overline\nabla_{\Upsilon^\pm} f}^2_t\right)+\norm{\overline\nabla_{\Upsilon^\pm} f}^2_t
\]
and use the Schwarz inequality, the small constant/large constant inequality,  \eqref{eq:bar_gradient_error_estimate} and \eqref{eq:gradient_error_estimate}
to control the error terms in \eqref{eq:int_by_part_identity}.  We conclude
\begin{equation}
\label{eq:gradient_estimate}
  \norm{\overline\nabla f}^2_t\geq\mp t(\omega(\Upsilon^\pm)f,f)_t  -\int_{\partial\Omega}\mathcal{L}(\Upsilon)\abs{f}^2 dS+O(\norm{f}^2_t)
\end{equation}
Substituting \eqref{eq:gradient_estimate} and \eqref{eq:hessian_bound} into \eqref{eq:MKH}, we have
\begin{equation}
\label{eq:MKH_estimate}
  \norm{\dbar f}^2_t+\norm{\dbar^*_t f}^2_t\geq \pm t ((q-\omega(\Upsilon^\pm))f,f)_t+\int_{\partial\Omega}(\mathcal{L}(f,f)-\mathcal{L}(\Upsilon)\abs{f}^2)e^{-t\phi}dS+O(\norm{f}^2_t).
\end{equation}

We are now ready to prove the basic estimate (see Proposition 3.1 in \cite{Shaw85a} for the case where $\Omega$ is the annuli between two weakly-pseudoconvex domains).
\begin{prop}
\label{prop:basic_estimates}
  Let $M$ be an $n$-dimensional Stein manifold, $n\geq 2$, and let $\Omega$ be a bounded subset of $M$ with $C^3$ boundary satisfying weak $Z(q)$ for some $1\leq q\leq n-1$.
  \begin{enumerate}
    \item \label{item:closed_range_estimate} For any constant $\eps>0$ there exists  $t_\eps>0$ and a $C_\eps>0$
    such that for any $t\geq t_\eps$ and  $f\in L^2_{(0,q)}(\Omega,e^{-t\phi})\cap\dom{\dbar}\cap\dom{\dbar^*_t}$ we have
        \[
          \eps\left(\norm{\dbar f}^2_t+\norm{\dbar^*_t f}^2_t\right)+C_\eps\norm{f}^2_{t,W^{-1}}\geq \norm{f}^2_t
        \]
     where $\|\cdot\|_{t,W^{-1}}$ is the dual norm to $\|\cdot\|_{t,W^1}$.
    \item \label{item:solvability_estimate} There exist constants $C>0$ and $\tilde t>0$ such that for all $t\geq \tilde t$ and
    $f\in L^2_{(0,q)}(\Omega,e^{-t\phi})\cap\dom{\dbar}\cap\dom{\dbar^*_t}\cap(\mathcal{H}^q_t(\Omega))^\bot$  we have
        \[
          C\left(\norm{\dbar f}^2_t+\norm{\dbar^*_t f}^2_t\right)\geq\norm{f}^2_t.
        \]
    \item \label{item:connected_boundary_estimate} If $\partial\Omega$ is connected, then for any constant $\eps>0$ there exists  $t_\eps>0$ such that for all $t\geq t_\eps$ and
    $f\in L^2_{(0,q)}(\Omega,e^{-t\phi})\cap\dom{\dbar}\cap\dom{\dbar^*_t}$ we have
        \[
          \eps\left(\norm{\dbar f}^2_t+\norm{\dbar^*_t f}^2_t\right)\geq \norm{f}^2_t.
        \]
  \end{enumerate}
\end{prop}
\begin{proof}
  Let $U^\pm$, $U^0$, and $\Upsilon^\pm$ be as in Lemma \eqref{lem:upsilon_extension}.  Let $\chi^\pm$ and $\chi^0$ form a partition of unity subordinate to $U^\pm$ and $U^0$.  Given $f\in C^2_{(0,q)}(\overline\Omega)\cap\dom{\dbar^*_t}$, we define $f^\pm=\chi^\pm f$ and $f^0=\chi^0 f$.  Since $\partial\Omega$ is $C^3$, $\Upsilon^\pm$ are $C^2$, and hence \eqref{eq:MKH_estimate} holds for $f^\pm$.  By \eqref{eq:levi_bound} and property \eqref{item:Levi form_property} of Definition \ref{defn:weak_Zq}, the boundary term in \eqref{eq:MKH_estimate} is positive, so we have
  \[
    \norm{\dbar f^\pm}^2_t+\norm{\dbar^*_t f^\pm}^2_t\geq \pm t((q-\omega(\Upsilon^\pm))f^\pm,f^\pm)_t+O(\norm{f^\pm}^2_t).
  \]
  By property \eqref{item:nondegeneracy} of Lemma \ref{lem:upsilon_extension}, we have $\pm(q-\omega(\Upsilon^\pm))>0$.  Since $\Omega$ is bounded we know that $\pm(q-\omega(\tilde\Upsilon))\geq C_0$ for some constant $C_0>0$.  Furthermore,
  \[
    \norm{\dbar f^\pm}^2_t+\norm{\dbar^*_t f^\pm}^2_t\leq 2\norm{\dbar f}^2_t+2\norm{\dbar^*_t f}^2_t+O(\norm{f}^2_t),
  \]
  so
  \[
    2\norm{\dbar f}^2_t+2\norm{\dbar^*_t f}^2_t\geq t C_0\norm{f^\pm}^2_t+O(\norm{f}^2_t).
  \]

  Since $f_0$ is compactly supported in $\Omega$, we have G{\aa}rding's inequality
  \[
    \norm{f^0}^2_{t,W^1}\leq C_t(\norm{\dbar f^0}^2_t+\norm{\dbar^*_t f^0}^2_t+\norm{f^0}^2_t)
  \]
  for some constant $C_t>0$.  Using the duality between $W^1_0$ and $W^{-1}$ we have
  \[
    \norm{f^0}^2_t\leq\norm{f^0}_{t,W^{-1}}\sqrt{C_t(\norm{\dbar f^0}^2_t+\norm{\dbar^*_t f^0}^2_t+\norm{f^0}^2_t)}.
  \]
  Associating $\sqrt{C_t}$ with the $\norm{f^0}_{t,W^{-1}}$ and applying the standard small constant/large constant inequality, we have for any $s>0$
  \[
    \norm{f^0}^2_t\leq \frac{s}{2}C_t\norm{f^0}_{t,W^{-1}}^2+\frac{1}{2s}(\norm{\dbar f^0}^2_t+\norm{\dbar^*_t f^0}^2_t+\norm{f^0}^2_t).
  \]
  Subtracting $\frac{1}{2s}\norm{f^0}^2_t$ from both sides and multiplying by $2s$ we have
  \[
    (2s-1)\norm{f^0}^2_t\leq s^2 C_t\norm{f^0}_{t,W^{-1}}^2+\norm{\dbar f^0}^2_t+\norm{\dbar^*_t f^0}^2_t.
  \]
  Letting $s=\frac{1}{2}(1+t C_0)$, we have
  \[
    t C_0\norm{f^0}^2_t\leq \frac{1}{4}(1+t C_0)^2 C_t\norm{f}_{t,W^{-1}}^2+2\norm{\dbar f}^2_t+2\norm{\dbar^*_t f}^2_t+O(\norm{f}^2_t).
  \]
  Combining the estimates for $f^0$ and $f^\pm$, we conclude
  \[
    \frac{t C_0}{3}\norm{f}^2_t+O(\norm{f}^2_t)\leq\frac{1}{4}(1+t C_0)^2 C_t\norm{f}_{t,W^{-1}}^2+2\norm{\dbar f}^2_t+2\norm{\dbar^*_t f}^2_t.
  \]
  We can now choose $t$ sufficiently large so that
  \[
    \frac{t C_0}{6}\norm{f}^2_t+O(\norm{f}^2_t)\geq\frac{1}{\eps}\norm{f}^2_t
  \]
  and the estimate is complete.  Standard density results (see for example Lemma 4.3.2 in \cite{ChSh01}) complete the proof of part \eqref{item:closed_range_estimate}.  The proof of part \eqref{item:solvability_estimate} is completed in the same manner as Lemma 3.1 in \cite{Shaw85a}, after setting $\tilde{t}=\inf_{\eps>0}t_\eps$.

  When the boundary is connected, we note that $U^0=U^-=\emptyset$ (see Remark \ref{rem:connected_boundary}), so there is no need to estimate $f_0$.  Hence the $W^{-1}$ terms are not necessary, and part \eqref{item:connected_boundary_estimate} follows.
\end{proof}

We immediately have the standard consequences of such $L^2$ estimates.
\begin{thm}
\label{thm:L2_theory}
  Let $M$ be an $n$-dimensional Stein manifold, $n\geq 2$, and let $\Omega$ be a bounded subset of $M$ with $C^3$ boundary satisfying weak $Z(q)$ for some $1\leq q\leq n-1$.  Then there exists a constant $\tilde t>0$ such that for all $t>\tilde t$ we have
  \begin{enumerate}
    \item $\mathcal{H}^q_t(\Omega)$ is finite dimensional.  If $\partial\Omega$ is connected, then $\mathcal{H}^q_t(\Omega)=\set{0}$.

    \item The weighted $\dbar$-Laplacian $\Box_t^q$ has closed range in $L^2_{(0,q)}(\Omega,e^{-t\phi})$.

    \item The weighted $\dbar$-Neumann operator $N_t^q$ exists and is continuous.

    \item The operator $\dbar$ has closed range in $L^2_{(0,q)}(\Omega,e^{-t\phi})$ and $L^2_{(0,q+1)}(\Omega,e^{-t\phi})$.

    \item The operator $\dbar^*_t$ has closed range in $L^2_{(0,q)}(\Omega,e^{-t\phi})$ and $L^2_{(0,q-1)}(\Omega,e^{-t\phi})$.

    \item The canonical solution operators to $\dbar$ given by $\dbar^*_t N^q_t:L^2_{(0,q)}(\Omega,e^{-t\phi})\rightarrow L^2_{(0,q-1)}(\Omega,e^{-t\phi})$ and $N^q_t\dbar^*_t:L^2_{(0,q+1)}(\Omega,e^{-t\phi})\rightarrow L^2_{(0,q)}(\Omega,e^{-t\phi})$ are continuous.

    \item The canonical solution operators to $\dbar^*_t$ given by $\dbar N^q_t:L^2_{(0,q)}(\Omega,e^{-t\phi})\rightarrow L^2_{(0,q+1)}(\Omega,e^{-t\phi})$ and $N^q_t\dbar:L^2_{(0,q-1)}(\Omega,e^{-t\phi})\rightarrow L^2_{(0,q)}(\Omega,e^{-t\phi})$ are continuous.

    \item For every $f\in L^2_{(0,q)}(\Omega)\cap\ker\dbar\cap(\mathcal{H}^q_t(\Omega))^\bot$ there exists a $u\in L^2_{(0,q-1)}(\Omega)$ such that $\dbar u=f$.
  \end{enumerate}
\end{thm}

\section{Sobolev Estimates}
\label{sec:Sobolev}

In this section we will use elliptic regularization to obtain estimates in the $L^2$-Sobolev space $W^1$ when $\partial\Omega$ is connected.  The first author
obtained such estimates for $C^2$-pseudoconvex domains in \cite{Har09}.
In that paper, he used an exhaustion by smooth strictly pseudoconvex domains.
Although smooth $Z(q)$ domains can exhaust bounded weakly $Z(q)$ domains with connected boundaries, constructing
$\Upsilon$ on the exhaustion domains in such a way that the estimates are uniform may not be possible.
Hence, we will use elliptic regularization in the present paper. Our discussion follows
the argument in Section 3.3 of \cite{Str10}, focusing on steps where the reduced boundary regularity requires more careful estimates.

We will need two equivalent norms on $W^1(\Omega)$: the standard norm $\norm{u}^2_{W^1}=\norm{u}^2+\norm{\nabla u}^2$ and the weighted norm
$\norm{u}^2_{t,W^1}=\norm{u}^2_t+\norm{\nabla u}^2_t$.  Although these are equivalent, the constant involved will depend on $t$, so for estimates where the dependency on $t$ is significant we will need
to use the weighted norm.  Only at the end of the proof will we be able to pass to estimates for the standard norm, which is more suitable for interpolation.
For $u\in W^1_{(0,q)}(\Omega)\cap\dom{\dbar^*_t}$ and $\delta>0$ we define
\[
  Q_{t,\delta}(u,u)=\norm{\dbar u}^2_t+\norm{\dbar^*_t u}^2_t+\delta\norm{\nabla u}^2_t.
\]
As in \cite{Str10}, we have a unique self-adjoint operator $\Box^q_{t,\delta}$ on $L^2_{(0,q)}(\Omega)$ satisfying $(\Box^q_{t,\delta}u,v)_t=Q_{t,\delta}(u,v)$ for all $u\in\dom{\Box^q_{t,\delta}}$ and
$v\in W^1_{(0,q)}(\Omega)\cap\dom{\dbar^*_t}$, where $\dom{\Box^q_{t,\delta}}$ is the subspace of $W^1_{(0,q)}(\Omega)\cap\dom{\dbar^*_t}$ on which
$\tilde{\Box}^q_{t,\delta} u\in L^2_{(0,q)}(\Omega)$ and $\tilde{\Box}^q_{t,\delta}$ is the canonical identification between $W^1_{(0,q)}(\Omega)\cap\dom{\dbar^*_t}$ and its conjugate dual.
We also obtain a unique solution operator $N^q_{t,\delta}$ mapping $L^2_{(0,q)}(\Omega)$ onto $\dom{\Box^q_{t,\delta}}$ satisfying $(u,v)_t=Q_{t,\delta}(N^q_{t,\delta} u,v)$.

By Proposition 3.5 in \cite{Str10}, $N^q_{t,\delta}$ maps $L^2_{(0,q)}(\Omega)$ continuously to $W^2_{(0,q)}(\Omega)$.
Although this proposition is stated for smooth domains, the proof for the $s=0$ case  holds on $C^3$ domains and we now outline the key step to illustrate the role of boundary smoothness.
Let $\rho$ be the signed distance function for $\Omega$, so that $\rho$ is a $C^3$ defining function \cite{KrPa81}, and let $(x_1,\ldots,x_{2n-1})$ be coordinates on $\partial\Omega$ with $\rho$ as the transverse coordinate.
Similarly, we choose an orthonormal basis for $(1,0)$-forms consisting of $\omega_1,\ldots,\omega_n$, where $\omega_n=\partial\rho$.  If we express $u$ in this basis, then the components will involve first derivatives of $\rho$.
If we let $D^h_j$ denote a difference quotient with respect to $x_j$, we can define $D^h_j u$ by considering difference quotients of components of $u$ in our special basis.  This will preserve $\dom{\dbar^*}$, but uniform bounds on $D^h_j u$ will now involve the $C^2$ norm of $\rho$.  Finally we wish to estimate $Q_{t,\delta}(D^h_j u,v)$.  The details for this estimate are contained in (3.38) through (3.41) in \cite{Str10}, but we will simply observe that they involve uniform bounds for $[\dbar,D^h_j]u$, $[\dbar^*_t,D^h_j]u$, and $[\nabla,D^h_j]u$, which will all involve the $C^3$ norm of $\rho$.  Working with the smooth cutoff functions necessary to work locally will not involve additional derivative of $\rho$, so Straube's argument will allow us to bound tangential derivatives of $u$ in the $W^1$ norm.  As usual, the structure of $\Box^q_{t,\delta}$ as a second-order elliptic operator will allow us to estimate the second derivatives in the normal direction (which only involve second derivatives of $\rho$).

Our goal is to show that the $W^1$ norm of $N^q_{t,\delta} f$ is bounded by the $W^1$ norm of $f$ with a constant that is independent of $\delta$, so that we may use a limiting argument to show that this estimate also holds for $N^q_t f$.  Since $\dbar\oplus\dbar^*_t$ is an elliptic system, it will suffice to estimate tangential derivatives, but first we must clarify how a differential operator acts globally on a $(0,q)$-form.  Let $\psi\in W^2_{(0,q)}(\Omega)\cap\dom{\dbar^*_t}$ and let $T$ be a differential operator defined on $\overline\Omega$ that is tangential on the boundary.  Since $\partial\Omega$ is $C^3$, we may assume $T$ has $C^2$ coefficients.  For $U\subset M$ with local coordinates $\set{z_1,\ldots,z_n}$ we can write
\[
  \psi=\sum_{I\in\mathcal{I}_q}\psi_I d\bar{z}_I,
\]
where the components $\psi_I$ are all in $W^2(\Omega)$.  The covariant derivative $\nabla_T\psi$ is globally defined and in local coordinates we can write
\[
  \nabla_T\psi=\sum_{I\in\mathcal{I}_q}T\psi_I d\bar{z}_I+O(\psi),
\]
where the coefficients in the zero order term involve coefficients of $T$ and are hence $C^2$.  However, $\nabla_T\psi$ is probably not in the domain of $\dbar^*_t$.  On the other hand, for $U$ sufficiently small, we can also choose a $C^2$ orthonormal basis for the space of $(1,0)$-forms $\omega_1,\ldots,\omega_n$ where $\omega_n=\partial\rho$.  In this basis we write
\[
  \psi=\sum_{I\in\mathcal{I}_q}\tilde{\psi}_I \bar\omega_I.
\]
Since the transition matrices between $dz$ and $\omega$ have $C^2$ entries, $\tilde{\psi}_I$ can be obtained by applying a linear operator with $C^2$ coefficients to $\psi_I$.  We define
\[
  D^U_T\psi=\sum_{I\in\mathcal{I}_q}T\tilde{\psi}_I \omega_I,
\]
and sum over a partition of unity to obtain $D_T\psi$.  Note that $D_T$ preserves tangential and normal components of $\psi$, so it will also preserve the domain of $\dbar^*_t$.  However, returning to local coordinates,
\[
  D^U_T\psi=\sum_{I\in\mathcal{I}_q}T\psi_I d\bar{z}_I+O(\psi),
\]
where the coefficients of the zero order terms are obtained by differentiating the $C^2$ transitions matrices between $dz$ and $\omega$, so they are only $C^1$.  Hence, $D_T-\nabla_T$ is a zero-order operator with $C^1$ coefficients.  This requires some caution.  For example, if $\nabla_D^2$ is a second-order differential operator with $C^1$ coefficients, then $[\nabla_D^2,\nabla_T]$ is a second-order differential operator with continuous coefficients, and hence $[\nabla_D^2,\nabla_T]\psi$ is a form in $L^2_{(0,q)}$.  However, $[\nabla_D^2,D_T-\nabla_T]$ may not be a differential operator with continuous coefficients, so we cannot make use of $[\nabla_D^2,D_T]\psi$.  On the other hand, commutators between $D_T$ and first-order differential operators will still have continuous (hence bounded) coefficients.

For $\eps>0$, let $t_\eps$ and $t$ be as in \eqref{item:connected_boundary_estimate} of Proposition \ref{prop:basic_estimates}.  Then for $u\in W^1_{(0,q)}(\Omega)\cap\dom{\dbar^*_t}$, when $\partial\Omega$ is connected we have $\eps Q_{t,\delta}(u,u)\geq \norm{u}^2_t$.  Let $f\in W^1_{(0,q)}(\Omega)$.  We immediately obtain
\begin{equation}
\label{eq:N_delta_L2_estimate}
  \norm{N^q_{t,\delta}f}_t\leq\eps\norm{f}_t.
\end{equation}
Since $N^q_{t,\delta}f\in W^2_{(0,q)}$, we can set $u=D_T N^q_{t,\delta} f$ and obtain $u\in W^1_{(0,q)}(\Omega)\cap\dom{\dbar^*_t}$.  Hence,
\begin{equation}
\label{eq:T_N_delta_first_substitution}
  \norm{D_T N^q_{t,\delta} f}^2_t\leq \eps Q_{t,\delta}(D_T N^q_{t,\delta} f,D_T N^q_{t,\delta} f).
\end{equation}

To estimate $Q_{t,\delta}(D_T N^q_{t,\delta} f,D_T N^q_{t,\delta} f)$, we will need to work with slightly smoother forms.  To that end, we introduce the following density lemma:
\begin{lem}
\label{lem:density}
  Let $\Omega\subset M$ be a bounded domain with $C^3$ boundary, and let $u\in W^1_{(0,q)}(\Omega)\cap\dom{\dbar^*_t}$.  Then there exists a sequence $u_\ell\in C^2_{(0,q)}(\overline\Omega)\cap\dom{\dbar^*_t}$ converging to $u$ in the $W^1$ norm.
\end{lem}
\begin{proof}
  Let $\rho$ be a $C^3$ defining function for $\Omega$, and let
  \[
    I_q=\set{f\in L^2_{(0,q)}(\Omega):f=\dbar\rho\wedge g,g\in L^2_{(0,q-1)}(\Omega)}.
  \]
Choose $\chi\in C^\infty_0(M)$ such that $\chi\equiv 1$ in a neighborhood of $\partial\Omega$ and $\dbar\rho\neq 0$ on the support of $\chi$.
If we let $\nu$ denote the orthogonal projection onto $I_q$ where defined, then $\chi\nu$ is a linear operator with $C^2(\overline\Omega)$ coefficients.  Since $u\in W^1_{(0,q)}(\Omega)$,
$u$ has a boundary trace in $L^2$.  By the usual density lemma (e.g., Lemma 4.3.2 in \cite{ChSh01}) and the usual characterization of $\dom{\dbar^*_t}$ (e.g., Lemma 4.2.1 in \cite{ChSh01}), we have that
the boundary trace of $\chi\nu u$ is zero a.e.  Since components of $\chi\nu u$ are in $W^1_0(\Omega)$, $\chi\nu u$ is the limit in $W^1$ of a sequence $u^\nu_\ell\in C^\infty_{0,(0,q)}(\Omega)$, so we can write
$u_\ell^\nu\rightarrow \chi\nu u$ in $W^1_{(0,q)}(\Omega)$.  We can also write $u$ as a limit in $W^1_{(0,q)}(\Omega)$ of forms
$\tilde u_\ell\in C^\infty_{(0,q)}(\overline\Omega)$.  If we set $u_\ell=u_\ell^\nu+(1-\chi\nu)\tilde u_\ell$, then $u_\ell\rightarrow u$ in $W^1_{(0,q)}(\Omega)$ since
  \[
    \norm{u_\ell-u}_{W^1}\leq\norm{u_\ell^\nu-\chi\nu u}_{W^1}+\norm{(1-\chi\nu)(\tilde u_\ell-u)}_{W^1}.
  \]
  Since $\nu$ has $C^2(\overline\Omega)$ coefficients, $u_\ell$ is in $C^2_{(0,q)}(\overline\Omega)$.
  Furthermore, since $\nu(1-\chi\nu)=(1-\chi)\nu$, we have $\nu u_\ell=0$ on the boundary of $\Omega$, so $u_\ell\in\dom{\dbar^*_t}$.
\end{proof}

With this density lemma in place, we are ready to prove the key lemma for our estimate (analogous to (3.50) in \cite{Str10}).
\begin{lem}\label{lem:elliptic_regularization_key_step}
  Let $\Omega\subset M$ be a bounded domain with connected $C^3$ boundary satisfying weak $Z(q)$.  There exist constants $C>0$ independent of $t$ and $C_t>0$ depending on $t$ such that for any $f\in W^1_{(0,q)}(\Omega)$ and differential operator $T$ with $C^2(\overline\Omega)$ coefficients that is tangential on the boundary of $\Omega$, we have
  \begin{equation}
  \label{eq:elliptic_regularization_key_step}
    Q_{t,\delta}(D_T N^q_{t,\delta} f,D_T N^q_{t,\delta} f)\leq C(D_T f,D_T N^q_{t,\delta} f)_t+C\norm{N^q_{t,\delta}f}^2_{t,W^1}
    +C_t\norm{f}_t^2.
  \end{equation}
\end{lem}
\begin{proof}
  We will adopt the convention that the values of $C>0$ and $C_t>0$ may increase from line to line.  Let $u=D_T N^q_{t,\delta} f$ and let $u_\ell\in C^2_{(0,q)}(\overline\Omega)\cap \dom{\dbar^*_t}$ be the sequence converging to $D_T N^q_{t,\delta}f$ given by Lemma \ref{lem:density}.  Note that the principal part of $\dbar^*_t$ is the same as the principal part of $\dbar^*$, so only the low order terms depend on $t$.  Hence, $[D_T,\dbar^*_t]$ has a first order component which is independent of $t$ and a lower order component which depends on $t$, so
  \[
  \begin{split}
    \norm{[D_T,\dbar^*_t] N^q_{t,\delta}f}_t&\leq C\norm{N^q_{t,\delta}f}_{t,W^1}+C_t\norm{N^q_{t,\delta}f}_t\\
    &\leq C\norm{N^q_{t,\delta}f}_{t,W^1}+C_t\norm{f}_t,
  \end{split}
  \]
  where the second inequality follows from \eqref{eq:N_delta_L2_estimate}.  Estimating commutators in this fashion gives us
  \begin{multline*}
    Q_{t,\delta}(u,u_\ell)\leq(D_T\dbar N^q_{t,\delta} f,\dbar u_\ell)_t+(D_T\dbar^*_t N^q_{t,\delta} f,\dbar^*_t u_\ell)_t
    +\delta(D_T\nabla  N^q_{t,\delta} f,\nabla u_\ell)_t\\
    +\left(C\norm{N^q_{t,\delta}f}_{t,W^1}+C_t\norm{f}_t\right)\sqrt{Q_{t,\delta}(u_\ell,u_\ell)}.
  \end{multline*}
  Since $T$ is tangential, we can integrate by parts and commute again to obtain
  \begin{multline*}
    Q_{t,\delta}(u,u_\ell)\leq Q_{t,\delta}(N^q_{t,\delta} f,(D_T)^*_t u_\ell)+(\dbar N^q_{t,\delta} f,[(D_T)^*_t,\dbar] u_\ell)_t\\+(\dbar^*_t N^q_{t,\delta} f,[(D_T)^*_t,\dbar^*_t]u_\ell)_t
    +\delta(\nabla  N^q_{t,\delta} f,[(D_T)^*_t,\nabla] u_\ell)_t\\
    +\left(C\norm{N^q_{t,\delta}f}_{t,W^1}+C_t\norm{f}_t\right)\sqrt{Q_{t,\delta}(u_\ell,u_\ell)}.
  \end{multline*}
  By definition,
  \[
    Q_{t,\delta}(N^q_{t,\delta} f,(D_T)^*_t u_\ell)=(f,(D_T)^*_t u_\ell)_t=(D_T f,u_\ell)_t.
  \]
  Since all terms with $u_\ell$ can now be estimated by the $W^1$ norm of $u_\ell$, we can take limits and obtain
  \begin{multline*}
    Q_{t,\delta}(u,u)\leq(D_T f,u)_t+(\dbar N^q_{t,\delta} f,[(D_T)^*_t,\dbar] u)_t+(\dbar^*_t N^q_{t,\delta} f,[(D_T)^*_t,\dbar^*_t]u)_t
    \\+\delta(\nabla  N^q_{t,\delta} f,[(D_T)^*_t,\nabla] u)_t
    +\left(C\norm{N^q_{t,\delta}f}_{t,W^1}
    +C_t\norm{f}_t\right)\sqrt{Q_{t,\delta}(u,u)}.
  \end{multline*}
  Now, we may also use the estimate
  \begin{multline*}
    \left(C\norm{N^q_{t,\delta}f}_{t,W^1}+C_t\norm{f}_t\right)\sqrt{Q_{t,\delta}(u,u)}\leq\\
    \frac{1}{2}\left(C\norm{N^q_{t,\delta}f}_{t,W^1}+C_t\norm{f}_t\right)^2+\frac{1}{2}Q_{t,\delta}(u,u)
  \end{multline*}
  and absorb the last term in the left-hand side to obtain
  \begin{multline*}
    Q_{t,\delta}(u,u)\leq 2(D_T f,u)_t+2(\dbar N^q_{t,\delta} f,[(D_T)^*_t,\dbar] u)_t+2(\dbar^*_t N^q_{t,\delta} f,[(D_T)^*_t,\dbar^*_t]u)_t
    \\+2\delta(\nabla  N^q_{t,\delta} f,[(D_T)^*_t,\nabla] u)_t
    +C\norm{N^q_{t,\delta}f}^2_{t,W^1}
    +C_t\norm{f}^2_t.
  \end{multline*}
  The remaining commutators will each be estimated using the same technique; we will illustrate the method with $(\dbar^*_t N^q_{t,\delta} f,[(D_T)^*_t,\dbar^*_t]u)_t$ since it could potentially involve additional factors of $t$.  Recall that $\nabla_T$ has $C^2$ coefficients, while $D_T-\nabla_T$ is a zero-order operator with $C^1$ coefficients.  Since adjoints of zero-order operators won't introduce additional factors of $t$, we can break down the commutator as follows:
  \begin{equation}\label{eqn:commute D_T* and dbar_t*}
    [(D_T)^*_t,\dbar^*_t]=[(D_T-\nabla_T)^*_t,\dbar^*]+[(D_T-\nabla_T)^*_t,(\dbar^*_t-\dbar^*)]+[(\nabla_T)^*_t,\dbar^*_t],
  \end{equation}
  The zero-order component of $[(D_T-\nabla_T)^*_t,\dbar^*]$ is independent of $t$ but only has continuous coefficients; we denote this $A$.  If we let $B_t=[(D_T)^*_t,\dbar^*_t]-A$, then $B_t$ is a first-order operator depending on $t$ (in the zero-order component) with $C^1$ coefficients.  Thus, if we commute $D_T$ with $B_t$ but not $A$ we obtain
  \[
  \begin{split}
    (\dbar^*_t N^q_{t,\delta} f,[(D_T)^*_t,\dbar^*_t]u)_t&=(\dbar^*_t N^q_{t,\delta} f,(A+B_t) D_T N^q_{t,\delta} f)_t\\&\leq(\dbar^*_t N^q_{t,\delta} f,D_T B_t N^q_{t,\delta} f)_t+C\norm{N^q_{t,\delta} f}^2_{t,W^1}+C_t\norm{f}^2_t.
  \end{split}
  \]
  Integrating $D_T$ by parts, commuting with $\dbar^*_t$, and using (\ref{eqn:commute D_T* and dbar_t*}), we have
  \[
    (\dbar^*_t N^q_{t,\delta} f,[(D_T)^*_t,\dbar^*_t]u)_t\leq(\dbar^*_t (D_T)^*_t N^q_{t,\delta} f,B_t N^q_{t,\delta} f)_t+C\norm{N^q_{t,\delta} f}^2_{t,W^1}+C_t\norm{f}^2_t.
  \]
  However, this can be bounded by
  \begin{multline*}
    (\dbar^*_t N^q_{t,\delta} f,[(D_T)^*_t,\dbar^*_t]u)_t\leq\sqrt{Q_{t,\delta}((D_T)^*_t N^q_{t,\delta} f,(D_T)^*_t N^q_{t,\delta} f)}\\
    \times\left(C\norm{N^q_{t,\delta} f}_{t,W^1}+C_t\norm{f}_t\right)+C\norm{N^q_{t,\delta} f}^2_{t,W^1}+C_t\norm{f}^2_t.
  \end{multline*}
  The same upper bound will be obtained if $\dbar^*_t$ is replaced with $\dbar$ or $\nabla$, so we have
  \begin{multline}
  \label{eq:almost_elliptic_regularization_estimate}
    Q_{t,\delta}(u,u)\leq \sqrt{Q_{t,\delta}((D_T)^*_t N^q_{t,\delta} f,(D_T)^*_t N^q_{t,\delta} f)}\left(C\norm{N^q_{t,\delta} f}_{t,W^1}+C_t\norm{f}_t\right)\\
    +2(D_T f,u)_t+C\norm{N^q_{t,\delta}f}^2_{t,W^1}
    +C_t\norm{f}^2_t.
  \end{multline}
  Note that $(D_T)^*_t=-D_T+a_t$, where $a_t$ is a $C^1$ function depending on $t$. In particular,
  \begin{multline*}
    Q_{t,\delta}((D_T)^*_t N^q_{t,\delta} f+u,(D_T)^*_t N^q_{t,\delta} f+u)=Q_{t,\delta}(a_t N^q_{t,\delta} f,a_t N^q_{t,\delta} f)\\
    \leq(f,\abs{a_t}^2 N^q_{t,\delta}f)_t+C_t\norm{f}_t\sqrt{Q_{t,\delta}(a_t N^q_{t,\delta} f,a_t N^q_{t,\delta} f)}.
  \end{multline*}
  After using the small constant/large constant inequality to absorb the $Q_{t,\delta}$ term on the right-hand side, we have
  \[
    Q_{t,\delta}((D_T)^*_t N^q_{t,\delta} f+u,(D_T)^*_t N^q_{t,\delta} f+u)\leq C_t\norm{f}^2_t,
  \]
  and hence
  \[
    Q_{t,\delta}((D_T)^*_t N^q_{t,\delta} f,(D_T)^*_t N^q_{t,\delta} f)\leq C Q_{t,\delta}(u,u)+C_t\norm{f}^2_t.
  \]
Substituting this into \eqref{eq:almost_elliptic_regularization_estimate} and again using the small constant/large constant estimate to absorb the
$Q_{t,\delta}$ term on the left-hand side will imply \eqref{eq:elliptic_regularization_key_step}.
\end{proof}
Combining Lemma \ref{lem:elliptic_regularization_key_step} and \eqref{eq:T_N_delta_first_substitution}, we have
\[
  \norm{D_T N^q_{t,\delta} f}^2_t\leq\eps C\left(\abs{(D_T f,D_T N^q_{t,\delta} f)_t}+\norm{N^q_{t,\delta}f}^2_{t,W^1}\right)+C_t\norm{f}_t^2.
\]
Another application of the small constant/large constant inequality to the first term on the right-hand side gives us
\[
  \norm{D_T N^q_{t,\delta} f}^2_t\leq\eps C\left(\norm{f}^2_{t,W^1}+\norm{N^q_{t,\delta}f}^2_{t,W^1}\right)+C_t\norm{f}_t^2.
\]
By Lemma 2.2 in \cite{Str10}, the normal derivatives of $N^q_{t,\delta}f$ consist of linear combinations of $\dbar N^q_{t,\delta}f$, $\dbar^* N^q_{t,\delta}f$, and tangential derivatives of $N^q_{t,\delta}f$.
We can convert $\dbar^*$ to $\dbar^*_t$ by adding a zeroth order term, so summing over all tangential derivatives and the normal derivative gives us
\[
  \norm{N^q_{t,\delta} f}^2_{t,W^1}\leq\eps C\left(\norm{f}^2_{t,W^1}+\norm{N^q_{t,\delta}f}^2_{t,W^1}\right)+C_t\norm{f}_t^2.
\]
When $\eps$ is sufficiently small (and hence $t$ is sufficiently large), we can absorb the $W^1$ norm of $N^q_{t,\delta}f$ on the left-hand side, obtaining
\begin{equation}
\label{eq:W1_estimate}
  \norm{N^q_{t,\delta} f}^2_{t,W^1}\leq\eps C\norm{f}^2_{t,W^1}+C_t\norm{f}_t^2.
\end{equation}
Since all constants have been chosen independently of $\delta$, the usual limiting argument will give us $W^1$ estimates for $N^q_t$.  Standard arguments now give us
\begin{thm}
\label{thm:Sobolev_Estimates}
  Let $M$ be an $n$-dimensional Stein manifold, and let $\Omega$ be a bounded subset of $M$ with connected $C^3$ boundary satisfying weak $Z(q)$ for some $1\leq q\leq n-1$.  Then there exists a constant $\tilde t>0$ such that for all $t>\tilde t$ and $-\frac{1}{2}\leq s\leq 1$ we have
  \begin{enumerate}
    \item \label{item:Sobolev_Neumann}The weighted $\dbar$-Neumann operator $N_t^q$ is continuous in $W^s_{(0,q)}(\Omega)$.

    \item \label{item:Sobolev_dbar_star_Neumann}The canonical solution operators to $\dbar$ given by\\ $\dbar^*_t N_t^q:W^s_{(0,q)}(\Omega)\rightarrow W^s_{(0,q-1)}(\Omega)$ and $N_t^q\dbar^*_t:W^s_{(0,q+1)}(\Omega)\cap\dom{\dbar^*_t}\rightarrow W^s_{(0,q)}(\Omega)$ are continuous.

    \item \label{item:Sobolev_dbar_Neumann}The canonical solution operators to $\dbar^*_t$ given by\\ $\dbar N_t^q:W^s_{(0,q)}(\Omega)\rightarrow W^s_{(0,q+1)}(\Omega)$ and $N_t^q\dbar:W^s_{(0,q-1)}(\Omega)\cap\dom{\dbar}\rightarrow W^s_{(0,q)}(\Omega)$ are continuous.

    \item \label{item:Sobolev_solve_dbar}For every $f\in W^s_{(0,q)}(\Omega)\cap\ker\dbar$ there exists a $u\in W^s_{(0,q-1)}(\Omega)$ such that $\dbar u=f$.
  \end{enumerate}
\end{thm}
\begin{rem} Theorem \ref{thm:Sobolev_Estimates} is identical to Theorem \ref{thm:main_theorem2} except for the additional hypothesis that $\partial\Omega$ is connected. Also,
$N^q_t\dbar^*_t$ is a problematic solution operator to $\dbar$ since it places a boundary condition on the data.
However, estimates for this operator are needed to obtain estimates for the dual
operator $\dbar N^q_t$ in the dual Sobolev space.
\end{rem}
\begin{proof}
We have already completed the proof of \eqref{item:Sobolev_Neumann} when $s=0$ and $s=1$, so interpolation will give us
the result for $0<s<1$.  We emphasize that the weighted Sobolev spaces used in \eqref{eq:W1_estimate} are equivalent to standard Sobolev spaces which are amenable to interpolation.
When $-\frac{1}{2}\leq s<0$ we use the duality between $W^s$ and $W^{-s}$ and the fact that $N^q_t$ is self-adjoint:
  \[
  \begin{split}
    \norm{N^q_t f}_{W^s}&\leq C_t\sup_{h\in W^{-s}(\Omega),\norm{h}_{W^{-s}=1}}\abs{(N^q_t f,h)_t}\\
    &\leq C_t\sup_{h\in W^{-s}(\Omega),\norm{h}_{W^{-s}=1}}\abs{(f,N^q_t h)_t}\leq C_t\norm{f}_{W^s}.
  \end{split}
  \]

  For $\dbar^*_t N^q_t$ and $\dbar N^q_t$, we use Lemma 3.2 in \cite{Str10}.  Although this Lemma assumes $N^q_t f$ is smooth, we can use elliptic regularization and Lemma \ref{lem:density} as before to obtain sufficient regularity.  Interpolation will give us estimates for $0\leq s\leq 1$.

  For $N^q_t\dbar$ and $N^q_t\dbar^*_t$, we first note that since $f\in W^1_{(0,q)}(\Omega)$ implies $N^q_t f$, $\dbar N^q_t f$, and $\dbar^*_t N^q_t f$ are all in $W^1(\Omega)$, we can conclude that
  \[
    \dbar D_T N^q_t f=[\dbar,D_T] N^q_t f+D_T\dbar N^q_t f\in L^2_{(0,q+1)}(\Omega)
  \]
  and
  \[
    \dbar^*_t D_T N^q_t f=[\dbar^*_t,D_T] N^q_t f+D_T\dbar^*_t N^q_t f\in L^2_{(0,q-1)}(\Omega).
  \]
  Thus we can substitute into \eqref{item:connected_boundary_estimate} of Proposition \ref{prop:basic_estimates} and estimate commutators to obtain
  \[
    \norm{D_T N^q_t f}^2_t\leq\eps\left(\norm{D_T\dbar N^q_t f}^2_t+\norm{D_T\dbar^*_t N^q_t f}^2_t+C\norm{N^q_t f}^2_{t,W^1}\right)+C_t\norm{N^q_t f}^2_t.
  \]
  Estimating the normal derivatives using Lemma 2.2 in \cite{Str10} we can estimate the $W^1$ norm of $N^q_t f$, and hence the error term on the right hand side can be absorbed into the left for sufficiently small $\eps>0$.  Hence we have (after possibly increasing $t_\eps$)
  \[
    \norm{N^q_{t}f}^2_{t,W^1}\leq \eps\left(\norm{\dbar N^q_{t}f}^2_{t,W^1}+\norm{\dbar^*_t N^q_{t}f}^2_{t,W^1}\right)+C_t\norm{N^q_{t}f}_t^2.
  \]
  For $f_1\in W^2_{(0,q-1)}(\Omega)$,  let $f=\dbar f_1$ to obtain
  \[
    \norm{N^q_t\dbar f_1}^2_{t,W^1}\leq \eps\norm{\dbar^*_t N^q_t\dbar f_1}^2_{t,W^1}+C_t\norm{N^q_{t}\dbar f_1}_t^2.
  \]
  To estimate the projection, we use integration by parts:
  \begin{multline*}
    \norm{\dbar^*_t N^q_t\dbar f_1}^2_{t,W^1}\leq(N^q_t\dbar f_1,\dbar f_1)_{t,W^1}+C\norm{N^q_t\dbar f_1}_{t,W^1}\norm{\dbar^*_t N^q_t\dbar f_1}_{t,W^1}\\
    +C_t\norm{N^q_t\dbar f_1}_t\norm{\dbar^*_t N^q_t\dbar f_1}_{t,W^1}+C_t\norm{N^q_t\dbar f_1}_{t,W^1}\norm{\dbar^*_t N^q_t\dbar f_1}_t.
  \end{multline*}
  Repeated use of the small constant/large constant inequality allows us to absorb the $\norm{\dbar^*_t N^q_t\dbar f_1}^2_{t,W^1}$ terms on the left-hand side and obtain
  \begin{multline*}
    \norm{\dbar^*_t N^q_t\dbar f_1}^2_{t,W^1}\leq C(N^q_t\dbar f_1,\dbar f_1)_{t,W^1}+C\norm{N^q_t\dbar f_1}_{t,W^1}^2\\
    +C_t\norm{N^q_t\dbar f_1}_t^2+C_t\norm{\dbar^*_t N^q_t\dbar f_1}_t^2.
  \end{multline*}
  A second integration by parts gives us
  \begin{multline*}
    \norm{\dbar^*_t N^q_t\dbar f_1}^2_{t,W^1}\leq C(\dbar^*_t N^q_t\dbar f_1,f_1)_{t,W^1}+C\norm{N^q_t\dbar f_1}_{t,W^1}\norm{f_1}_{t,W^1}\\
    +C_t\norm{N^q_t\dbar f_1}_t\norm{f_1}_{t,W^1}+C\norm{N^q_t\dbar f_1}_{t,W^1}^2
    +C_t\norm{N^q_t\dbar f_1}_t^2+C_t\norm{\dbar^*_t N^q_t\dbar f_1}_t^2,
  \end{multline*}
  and more small constant/large constant inequalities yield
  \begin{multline*}
    \norm{\dbar^*_t N^q_t\dbar f_1}^2_{t,W^1}\leq C\norm{f_1}^2_{t,W^1}+C\norm{N^q_t\dbar f_1}_{t,W^1}^2\\
    +C_t\norm{f_1}^2_t+C_t\norm{N^q_t\dbar f_1}_t^2+C_t\norm{\dbar^*_t N^q_t\dbar f_1}_t^2.
  \end{multline*}
  All of the terms with $C_t$ in front are bounded by $\norm{f_1}^2_t$, so
  \[
    \norm{\dbar^*_t N^q_t\dbar f_1}^2_{t,W^1}\leq C\norm{f_1}^2_{t,W^1}+C\norm{N^q_t\dbar f_1}_{t,W^1}^2
    +C_t\norm{f_1}^2_t.
  \]
  Now we can substitute into the original estimate for $\norm{N^q_t\dbar f_1}_{t,W^1}^2$ and, by making $\eps$ sufficiently small and adjusting $t_\eps$ we obtain
  \[
    \norm{N^q_t\dbar f_1}^2_{t,W^1}\leq \eps\norm{f_1}^2_{t,W^1}+C_t\norm{f_1}_t^2.
  \]
  For $f_2\in W^2_{(0,q+1)}(\Omega)\cap\dom{\dbar^*_t}$, let $f=\dbar^*_t f_2$ and use similar techniques to estimate $N^q_t\dbar^*_t f_2$.  Density lets us generalize to $f_1$ and $f_2$ in $W^1$.

  As before, we interpolate to obtain estimates for $0<s<1$ and use duality to obtain estimates for $-\frac{1}{2}\leq s<0$ (since we now have estimates for the adjoint of each operator).
\end{proof}

\section{Proofs of Main Theorems}
\label{sec:proofs_main_theorems}

When $s<\frac{1}{2}$, we immediately obtain a solution operator for the $\dbar$-Cauchy Problem (see Section 9.1 in \cite{ChSh01}).
\begin{cor}
\label{cor:compact_support_solvability}
  Let $M$ be an $n$-dimensional Stein manifold, let $\Omega$ be a bounded subset of $M$ with connected $C^3$ boundary satisfying weak $Z(n-q-1)$ for some $1\leq q<n-1$, and let $-\frac{1}{2}\leq s<\frac{1}{2}$. For every $f\in W^s_{(0,q+1)}(M)\cap\ker\dbar$ supported in $\overline\Omega$ there exists a $u\in W^s_{(0,q)}(M)$ supported in $\overline\Omega$ such that $\dbar u=f$, and the solution operator is continuous.
\end{cor}
\begin{proof}
  The proof is identical to that given for Proposition 3.4 in \cite{Shaw03}.  It suffices to have estimates for $\dbar N^{(n,n-q-1)}$.  On $Z(n-q-1)$ domains we have estimates for the operator $\dbar N^{(0,n-q-1)}$, but these are equivalent since the holomorphic component of $(p,q)$-forms has no impact on our estimates (except in the curvature terms, but these are all dominated when $t$ is large).
\end{proof}

Furthermore, we can now use techniques of Shaw \cite{Shaw10} to remove the requirement that the boundary of $\Omega$ be connected.
\begin{cor}
  Let $M$ be an $n$-dimensional Stein manifold, and let $\Omega$ be a bounded subset of $M$ with $C^3$ boundary satisfying weak $Z(q)$ for some $1\leq q<n-1$.  For every $f\in L^2_{(0,q)}(\Omega)\cap\ker\dbar$ there exists a $u\in L^2_{(0,q-1)}(\Omega)$ such that $\dbar u=f$.  In particular, $\mathcal{H}^q(\Omega)=\set{0}$.
\end{cor}
\begin{rem}
  This is also known for bounded weak $Z(n-1)$ domains, by Remark \ref{rem:Z(n-1)} and Theorem \ref{thm:L2_theory}.
\end{rem}
\begin{proof}
  Without loss of generality, assume that $\Omega$ is connected.  By Lemma \ref{lem:connected_decomposition} we have the decomposition $\Omega=\Omega_1\backslash\bigcup_{j=2}^m\overline\Omega_j$ where $\Omega_1$ is a weak $Z(q)$ domain with connected boundary and each $\Omega_j$ with $2\leq j\leq m$ is a weak $Z(n-q-1)$ domain with connected boundary and $\overline\Omega_j\subset\Omega_1$.

The proof now follows like the proof of Theorem 3.2 in \cite{Shaw10}, but we must make a few clarifying remarks.
Shaw's proof initially finds a solution $U$ in $W^{-1}(\Omega_1)$, and then proposes two techniques for regularizing this to a solution in $L^2(\Omega_1)$.
The first technique requires approximating $\Omega_1$ from within by smooth strongly pseudoconvex domains.  While we can approximate from within by smooth $Z(q)$ domains,
the role of $\Upsilon$ on these domains is unpredictable, so the constants in our estimates may not be uniform.
Hence, we  use the second technique, which decomposes $U$ into a component supported near the boundary of $\Omega_1$ and a compactly supported component.
This technique requires solving $\dbar$ for compactly supported forms in $\Omega_1$, but as Shaw points out, this can be accomplished by solving $\dbar$ on a ball containing $\Omega_1$.

This gives us vanishing of $\mathcal{H}^q_t(\Omega)$.  However, this is isomorphic to a Dolbeault cohomology group which is independent of $t$,
so the dimension of $\mathcal{H}^q_t(\Omega)$ must also be independent of $t$.  Hence, $\mathcal{H}^q(\Omega)=\set{0}$ as well.
\end{proof}

Now that the vanishing of the space of harmonic $(0,q)$-forms has been established, Theorem \ref{thm:main_theorem1} will follow for the unweighted operators and spaces (see Theorem 4.4.1 in \cite{ChSh01}) from Theorem \ref{thm:L2_theory}.

In the proof of Theorem \ref{thm:Sobolev_Estimates} we required a connected boundary only because it gave us vanishing for the space of harmonic $(0,q)$-forms.  Now that we've established this for disconnected boundaries, the proof of Theorem \ref{thm:Sobolev_Estimates} can be carried through for all bounded domains, as in the statement of Theorem \ref{thm:main_theorem2}.

In \cite{Shaw85}, Shaw uses the $\dbar$ operator to solve $\dbar_b$ extrinsically on smooth pseudoconvex domains.  Since we are working on $C^3$ domains, we will more closely follow the methods of \cite{Shaw03}, which are adapted to work on non-smooth domains. We now prove Theorem \ref{thm:main_theorem3}.
\begin{proof}[Proof of Theorem \ref{thm:main_theorem3}]
We outline the proof, following the construction in \cite{Shaw03}.  By embedding $M$ in $\mathbb{C}^{2n+1}$, we can pullback a ball containing the image of
$\partial\Omega$ to obtain a strictly pseudoconvex set $B$ such that $\overline\Omega\subset B$.  Let $\Omega^+=B\backslash\overline\Omega$ and $\Omega^-=\Omega$.
In \cite{HeLe81} a Martinelli-Bochner-Koppelman type kernel is constructed for Stein manifolds, and in \cite{Lau87} it is shown that the transformation induced by this kernel satisfies a jump formula.
As a result, there exists an integral kernel $K_q(\zeta,z)$ of type $(0,q)$ in $z$ and $(n,n-q-1)$ in $\zeta$ satisfying a Martinelli-Bochner-Koppelman formula such that we can define
  \[
    \int_{\partial\Omega} K_q(\zeta,z)\wedge f(\zeta)=\begin{cases}f^+(z) & z\in\Omega^+\\f^-(z) & z\in\Omega^-\end{cases}
  \]
  (see section 2.4 in \cite{HeLe81}).  If, by an abuse of notation, we extend $f^+(z)$ and $f^-(z)$ to $\partial\Omega$ by considering non-tangential limits, we have the jump formula
  \[
    \tau f^+(z)-\tau f^-(z)=(-1)^q f(z)
  \]
on $\partial\Omega$ where $\tau$ is defined in (\ref{eqn:tau def'n})
(see Proposition 2.3.1 in \cite{Lau87}).  Since $\dbar_b f=0$, we have $\dbar f^+=0$ and $\dbar f^-=0$.
Since $f^+$ and $f^-$ have non-tangential boundary values in $L^2$, they have components in $W^{1/2}$.

Let $E$ denote a linear extension operator from $\Omega^+$ to $B$ that is continuous in the Sobolev spaces $W^s(\Omega^+)$ for $0\leq s \leq 1/2$.
Applying this componentwise to $f^+$, we obtain a $(0,q)$-form $Ef^+$ that is in $W^{1/2}(B)$.
Since $\dbar f^+=0$ on $\Omega^+$, we have $\dbar E f^+$ supported in $\Omega^-$.  Furthermore, $\dbar E f^+$ has components in $W^{-1/2}(M)$, so by Corollary \ref{cor:compact_support_solvability}
there exists $V\in W^{-1/2}_{(0,q)}(M)$ supported in $\Omega^-$ satisfying $\dbar V=\dbar E f^+$.  Thus $\tilde f^+=E f^+-V\in W^{-1/2}_{(0,q)}(B)$ satisfies $\dbar\tilde{f}^+=0$ on $B$ and $\tilde{f}^+=f^+$
on $\Omega^+$.

We may now use the canonical solution operator to define $u^+=\dbar^* N^q_B \tilde f^+$ and $u^-=\dbar^*_t N^q_{t,\Omega} f^-$.  By interior regularity for $N^q_B$, $u^+$ gains a derivative on compact subsets of $B$, so $u^+$ has components in $W^{1/2}$ on a neighborhood of $\overline\Omega$.  By Theorem \ref{thm:Sobolev_Estimates}, $u^-$ also has components in $W^{1/2}$.  Since each of these forms satisfies an elliptic system, they have trace values in $L^2(\partial\Omega)$, so abusing notation we can write $u=(-1)^q(u^+-u^-)$ on $\partial\Omega$.
\end{proof}

\section{Examples}
\label{sec:example}

Our first goal in this section is to motivate Definition \ref{defn:weak_Zq} by constructing an example where this definition holds but earlier definitions fail.  In \cite{HaRa11} we proposed a
more restrictive definition for weak $Z(q)$.  In the context of bounded domains in $\mathbb{C}^n$ with connected boundaries, this was equivalent to $(q-1)$-pseudoconvexity in the sense of \cite{Zam08}.
The key difference is that in $(q-1)$-pseudoconvexity the eigenvalues of $\Upsilon$ must all be zero or one.  This can often be achieved locally by changing the metric.  The critical issue seems to be that if $
\Upsilon$ is continuous then $(q-1)$-pseudoconvexity forces the rank of $\Upsilon$ and the rank of $I-\Upsilon$ to be locally constant.  We will construct an example where this is impossible, but the added
flexibility of Definition \ref{defn:weak_Zq} still allows us to prove closed range results.  We summarize the first result of this section as follows:
\begin{prop}
  \label{prop:key_counterexample}
  There exists a smooth bounded domain $\Omega\subset\mathbb{C}^3$ with $0\in\partial\Omega$ such that:
  \begin{enumerate}
    \item There does not exist a hermitian metric in a neighborhood of the origin such that $\partial\Omega$ is 1-pseudoconvex at the origin.

    \item In the Euclidean metric, $\partial\Omega$ is weakly $Z(2)$.
  \end{enumerate}
\end{prop}

In what follows, we assume $M=\mathbb{C}^n$ is equipped with a strictly plurisubharmonic exhaustion function $\varphi$.  After adding a pluriharmonic polynomial to $\varphi$, we may assume $
\varphi(0)=0$, $d\varphi(0)=0$, and $\frac{\partial^2\varphi}{\partial z_j\partial z_k}(0)=0$ for all $1\leq j,k\leq n$.  Under these assumptions, $\varphi$ is strictly convex in a neighborhood of the origin, so $
\varphi(z)-R^2$ is a defining function for a strictly convex domain when $R$ is sufficiently small.

Let $\psi\in C^\infty(\mathbb{R})$ satisfy $\psi''\geq 0$ and $\psi(x)=\abs{x}$ for all $\abs{x}\geq 1$.  Note that this forces $|\psi'|\leq 1$.
Since our main results hold for $C^3$ domains, we could also use a $C^3$ function that is piecewise polynomial to obtain $C^3$ domains, but we will continue to assume that $\psi$ is $C^\infty$ in order to make clear that boundary smoothness does not play a role in our counterexamples.

For $r>0$ set $\psi_r(x)=r\psi\left(\frac{x}{r}\right)$, so that $\psi''_r\geq 0$ and $\psi_r(x)=\abs{x}$ for all $\abs{x}\geq r$.  Note that this forces $\abs{\psi'_r}\leq 1$.  Given any two defining function
$\rho_1$ and $\rho_2$, we wish to construct a smooth approximation to the defining function $\max\set{\rho_1,\rho_2}$.  Hence, we define
\[
  \rho(z)=\frac{1}{2}\psi_r(\rho_1(z)-\rho_2(z))+\frac{1}{2}(\rho_1(z)+\rho_2(z)).
\]
We compute
\begin{align*}
  \dbar\rho&=\frac{1}{2}(1+\psi'_r(\rho_1(z)-\rho_2(z)))\dbar\rho_1+\frac{1}{2}(1-\psi'_r(\rho_1(z)-\rho_2(z)))\dbar\rho_2\\
  \ddbar\rho&=\frac{1}{2}(1+\psi'_r(\rho_1(z)-\rho_2(z)))\ddbar\rho_1+\frac{1}{2}(1-\psi'_r(\rho_1(z)-\rho_2(z)))\ddbar\rho_2\\
  &\qquad+\frac{1}{2}\psi''_r(\rho_1(z)-\rho_2(z))(\partial\rho_1-\partial\rho_2)\wedge(\dbar\rho_1-\dbar\rho_2).
\end{align*}

Let $\rho_1(z)=-\im z_n+P(z_1,\ldots,z_{n-1})$, where $P$ is a smooth function vanishing at the origin.  Then $\dbar\rho_1=-\frac{i}{2} d\bar{z}_n+\dbar P$ and $\ddbar\rho_1=\ddbar P$.  Let $\rho_1$ be
the defining function for an unbounded domain $\Omega_1$.  For any $z\in\partial\Omega$ and $L\in T^{1,0}_z(\mathbb{C}^n)$,
 the mapping $L \mapsto L' :=  L+2i(\partial\rho_1(L))\frac{\partial}{\partial z_n}$
projects $L$ onto $L'\in T^{1,0}_z(\partial\Omega_1)$.
Note that $i\ddbar\rho_1(i\bar L\wedge L)=\mathcal{L}_{\rho_1}(i\bar L'\wedge L')$.  Hence, every eigenvector and eigenvalue of $\mathcal{L}_{\rho_1}$ corresponds
to an eigenvector and eigenvalue of $i\ddbar\rho_1$, and the additional eigenvalue of zero corresponds to the eigenvector $\frac{\partial}{\partial z_n}$.

Let $\rho_2(z)=\varphi(z)-R^2$, and recall that this defines a bounded strictly convex domain for $R>0$ sufficiently small.  Then for $r$ and $R$ sufficiently small, $\rho(z)=\rho_1(z)$ in a neighborhood of the origin, but $\rho(z)$ defines a bounded domain.  For any $L\in T^{1,0}(\partial\Omega)$, we check
\begin{multline*}
  \mathcal{L}_\rho(i\bar{L}\wedge L)\geq \\ \frac{1}{2}(1+\psi'_r(\rho_1(z)-\rho_2(z)))\mathcal{L}_{\rho_1}(i\bar{L}'\wedge L')+\frac{1}{2}(1-\psi'_r(\rho_1(z)-\rho_2(z)))\abs{L}^2.
\end{multline*}
Hence, positive eigenvalues will remain positive, while zero eigenvalues will become positive when $\psi'_r<1$.  We conclude
\begin{lem}
\label{lem:bounded_domain}
  Let $\Omega_1$ be an unbounded domain in $\mathbb{C}^n$ defined by the defining function $\rho_1(z)=-\im z_n+P(z_1,\ldots,z_{n-1})$ where $P$ is a smooth function vanishing at the origin.  There exists an arbitrarily small bounded domain $\Omega$ defined by a smooth defining function $\rho$ such that $\rho=\rho_1$ on a neighborhood of the origin and the Levi form of $\partial\Omega$ has at least as many positive eigenvalues as the Levi form of $\partial\Omega_1$ has nonnegative eigenvalues when $\rho\neq\rho_1$.
\end{lem}

Now, we are ready to introduce our unbounded domain.
\begin{proof}[Proof of Proposition \ref{prop:key_counterexample}]
  For convenience, we set $z_1=x+iy$ and define $P(z_1,z_2)=2x\abs{z_2}^2-x y^4$.  Let $\Omega_1\subset\mathbb{C}^3$ be defined by $\rho_1(z)=-\im z_3+P(z_1,z_2)$.  Since $\dbar x=\frac{1}{2}d\bar{z}_1$ and $\dbar y=\frac{i}{2}d\bar{z}_1$ we compute:
  \[
    \dbar\rho_1=\left(\abs{z_2}^2-\frac{1}{2}y^4-2i x y^3\right) d\bar{z}_1+2x z_2 d\bar{z}_2-\frac{i}{2} d\bar{z}_3
  \]
  and
  \begin{equation}
  \label{eq:example_levi_form}
    \ddbar\rho_1=-3x y^2 dz_1\wedge d\bar{z}_1+z_2 dz_1\wedge d\bar{z}_2+\bar{z}_2 dz_2\wedge d\bar{z}_1+2x dz_2\wedge d\bar{z}_2.
  \end{equation}
  We choose a basis for $T^{1,0}(\partial\Omega_1)$ by setting $L_j=\frac{\partial}{\partial z_j}+2i\frac{\partial P}{\partial z_j}\frac{\partial}{\partial z_3}$ for $1\leq j\leq 2$.  Under this basis we can represent the Levi form by the matrix $c_{j\bar k}=\mathcal{L}_{\rho_1}(i \bar L_k\wedge L_j)=i\ddbar\rho_1\left(i\frac{\partial}{\partial\bar z_k}\wedge \frac{\partial}{\partial z_j}\right)$.  Note that
  \[
  \det\left(
      \begin{array}{cc}
        -3x y^2 & z_2 \\
        \bar{z}_2 & 2x \\
      \end{array}
    \right)=-6x^2 y^2-\abs{z_2}^2.
  \]
  Since this is nonpositive, the Levi form always has one nonpositive and one nonnegative eigenvalue.  When $z_2\neq 0$ or both $x\neq 0$ and $y\neq 0$, then the determinant is strictly negative so the Levi form has one negative eigenvalue and one positive eigenvalue.  In this case, $Z(2)$ is satisfied.  If the determinant vanishes but $x(2-3y^2)>0$, then the Levi form has one positive eigenvalue and one zero eigenvalue, so $Z(2)$ still holds.

  Both 1-pseudoconvexity and the definition of weak $Z(2)$ given in \cite{HaRa11} require orthonormal coordinates.  For an arbitrary hermitian metric, let $u_1$ and $u_2$ be an orthonormal basis for $T^{1,0}(\partial\Omega_1)$.  Each of these can be written in the form
  \[
    u_j=\sum_{k=1}^2 a_j^k L_k
  \]
  for smooth functions $a_j^k$.  We use $c^u_{j\bar k}$ to denote the Levi form with respect to these new coordinates, and note that
  \begin{equation}
  \label{eq:Levi_transform}
    c^u_{j\bar k}=\sum_{\ell,m=1}^2 a_j^\ell c_{\ell\bar m}\bar{a}^m_k.
  \end{equation}
  The eigenvalues of $c^u$ will be denoted $\mu^u_1\leq\mu^u_2$.  Computing the trace, we have
  \[
    \mu^u_1+\mu^u_2=\sum_{k,\ell,m=1}^2 a_k^\ell c_{\ell\bar m}\bar{a}^m_k=\sum_{\ell,m=1}^2 g^{\bar m\ell}c_{\ell\bar m}
  \]
  where $g^{\bar m\ell}$ is a positive definite $2\times 2$ hermitian matrix.  Substituting \eqref{eq:example_levi_form}, we have
  \begin{equation}
    \label{eq:sum_two_eigenvalues}
    \mu^u_1+\mu^u_2=-3xy^2 g^{\bar 1 1}+2\re(z_2 g^{\bar 2 1})+2x g^{\bar 2 2}.
  \end{equation}

  For $\partial\Omega$ to be 1-pseudoconvex, we need either $\mu^u_1+\mu^u_2\geq 0$ or $\mu^u_1+\mu^u_2-c^u_{1\bar{1}}\geq 0$.  We will show that each of these leads to contradictions.

  First, assume that $\mu^u_1+\mu^u_2\geq 0$.  Set $y=0$, so that by \eqref{eq:sum_two_eigenvalues}, $2\re(z_2 g^{\bar 2 1})+2x g^{\bar 2 2}\geq 0$.  Since the left hand side of this inequality equals zero when $x=z_2=0$, this must be a critical point and hence all first derivatives in $x$ or $z_2$ will also vanish when $x=z_2=0$.  Hence, when $x=z_2=0$ we have $g^{\bar 2 1}=g^{\bar 2 2}=0$.  However, this implies that $g^{\bar m \ell}$ has rank 1, contradicting the fact that $g^{\bar m \ell}$ is nondegenerate.  Thus $\mu^u_1+\mu^u_2$ can not be nonnegative in a neighborhood of the origin.

  Now, we assume that $\mu^u_1+\mu^u_2-c^u_{1\bar{1}}\geq 0$.  Combining our assumption with \eqref{eq:sum_two_eigenvalues} yields
  \[
    -3xy^2(g^{\bar 1 1}-\abs{a_1^1}^2)+2\re(z_2(g^{\bar 2 1}-a_1^1\bar{a}_1^2))+2x(g^{\bar 2 2}-\abs{a_1^2}^2)\geq 0.
  \]
  As before, derivatives in $x$ and $z_2$ must vanish when $x=z_2=0$, so at these points we have $g^{\bar 2 1}-a_1^1\bar{a}_1^2=0$ and
  \[
    -3y^2(g^{\bar 1 1}-\abs{a_1^1}^2)+2(g^{\bar 2 2}-\abs{a_1^2}^2)=0
  \]
  When $y\neq 0$, this means that the rank of $g^{\bar k j}-\bar{a}_1^k a_1^j$ is either zero or two.  However,
  \[
    g^{\bar m\ell}-\bar{a}_1^m a_1^\ell=\sum_{j,k=1}^2\bar{a}_k^m (\delta_{\bar k j}-\delta_{j1}\delta_{k1})a_j^\ell.
  \]
  Hence $g^{\bar k j}-\bar{a}_1^k a_1^j$ must have rank one, contradicting the fact that it must have a rank of zero or two.  We conclude that $\mu^u_1+\mu^u_2-c^u_{1\bar{1}}$ can not be nonnegative in a neighborhood of the origin.

  Since the above construction was carried out for an arbitrary metric, we conclude that there is no metric in a neighborhood of the origin in which $\partial\Omega$ is 1-pseudoconvex.

  We must now show that Definition \ref{defn:weak_Zq} holds for this domain under the Euclidean metric (i.e., $\varphi=\abs{z}^2$).  If $u_1$ and $u_2$ are orthonormal vectors in the span of $L_1$ and $L_2$ then the sum of the two smallest eigenvalues of the Levi form are
  \[
    \mu_1+\mu_2=\mathcal{L}(i\bar{u}_1\wedge u_1+i\bar{u}_2\wedge u_2).
  \]
  If
  \[
    \Upsilon_t=i\bar u_1\wedge u_1+i\bar u_2\wedge u_2-t(2 i\bar L_1\wedge L_1+3 y^2 i\bar L_2\wedge L_2),
  \]
  then
  \[
    \mu_1+\mu_2-\mathcal{L}(\Upsilon_t)=0,
  \]
  so \eqref{item:Levi form_property} of Definition \ref{defn:weak_Zq} is satisfied.  Since $L_1$ and $L_2$ are also in the span of $u_1$ and $u_2$, condition \eqref{item:upper_and_lower_bounds} will follow for $t\geq 0$ sufficiently small.  Finally, $\omega(\Upsilon_t)=2-t(2\abs{L_1}^2+3y^2\abs{L_2}^2)$, so condition \eqref{item:trace} holds for any $t\neq 0$.  Hence $\Omega_1$ satisfies weak $Z(2)$.

  We use Lemma \ref{lem:bounded_domain} to turn our example into a bounded domain.  We have already shown that weak $Z(2)$ is satisfied in the interior of the set where $\rho=\rho_1$, and it is automatically satisfied at all points where $\rho\neq\rho_1$ because the usual definition of $Z(2)$ holds, so it remains to check the points on the boundary of the set where $\rho=\rho_1$.  At such points, let $\tilde{L}_1$ and $\tilde{L}_2$ be smooth extensions of $L_1$ and $L_2$.  If $\tilde{u}_1$ and $\tilde{u}_2$ are orthonormal vectors in the span of $\tilde{L}_1$ and $\tilde{L}_2$, we define
  \[
    \Upsilon_t=i \bar{ \tilde {u}}_1\wedge \tilde u_1+i\bar{ \tilde{ u}}_2\wedge \tilde u_2-t(2 i\bar{ \tilde{ L}}_1\wedge \tilde L_1+3 y^2 i\bar{ \tilde{ L}}_2\wedge \tilde L_2).
  \]
  We still have
  \[
    \mu_1+\mu_2-\mathcal{L}(\Upsilon_t)=t\mathcal{L}(2 i\bar{ \tilde{ L}}_1\wedge \tilde L_1+3 y^2 i\bar{ \tilde{ L}}_2\wedge \tilde L_2).
  \]
  This quantity is zero when $\rho=\rho_1$, and by construction we are adding a positive component to $\mathcal{L}$ when we transition to $\rho_2$, so $\mu_1+\mu_2-\mathcal{L}(\Upsilon_t)\geq 0$ in a neighborhood of the set where $\rho=\rho_1$.  Using the patching argument of Lemma \ref{lem:local}, a global $\Upsilon$ can be obtained.
\end{proof}

We will also construct an example demonstrating that our condition is not invariant under changes of metric.  Roughly speaking, this example contains a direction which is poorly behaved (at some points in a neighborhood of the origin this direction is an eigenvector of the Levi form with a negative eigenvalue, while at other points it can be an eigenvector corresponding to the largest positive eigenvalue).  In order for weak $Z(2)$ to be satisfied, the metric must be chosen to minimize the size of this direction (and prevent it from corresponding to the largest positive eigenvalue).
\begin{prop}
  \label{prop:metric_counterexample}
  There exists a bounded domain $\Omega\subset\mathbb{C}^4$ with smooth boundary such that
  \begin{enumerate}
    \item $\partial\Omega$ does not satisfy weak $Z(2)$ under the Euclidean metric.
    \item There exists a strictly plurisubharmonic exhaustion function $\varphi$ for $\mathbb{C}^4$ such that $\partial\Omega$ satisfies weak $Z(2)$ with respect to the metric $\omega=i\ddbar\varphi$.
  \end{enumerate}
\end{prop}
\begin{proof}
  As before, we can construct an unbounded domain and use Lemma \ref{lem:bounded_domain} to make this bounded.  To minimize the number of subscripts,
  we will write $\rho$ in place of the $\rho_1$ used in the statement of Lemma \ref{lem:bounded_domain}.  As above, let $z_1=x+iy$ and $\rho(z)=-\im z_4+P(z_1,z_2,z_3)$ where
  $P(z_1,z_2,z_3)=-9\abs{z_1}^4+6(x^2\abs{z_2}^2+y^2\abs{z_3}^2)+\abs{z_2}^2\abs{z_3}^2+\frac{1}{4}(\abs{z_2}^4+\abs{z_3}^4)$.
  We choose a basis for $T^{1,0}(\partial\Omega)$ by setting $L_j=\frac{\partial}{\partial z_j}+2i\frac{\partial P}{\partial z_j}\frac{\partial}{\partial z_4}$ for $1\leq j\leq 3$.

  We compute
  \begin{multline*}
    \dbar\rho=\left(-18\abs{z_1}^2 z_1+6x\abs{z_2}^2+6iy\abs{z_3}^2\right)d\bar{z}_1+\left(6x^2z_2+\abs{z_3}^2z_2+\frac{1}{2}\abs{z_2}^2z_2\right)d\bar{z}_2\\+\left(6y^2z_3+\abs{z_2}^2z_3+\frac{1}{2}\abs{z_3}^2z_3\right)d\bar{z}_3-\frac{i}{2}d\bar{z}_4
  \end{multline*}
  and
  \begin{multline*}
    i\ddbar\rho=i(-36\abs{z_1}^2+3\abs{z_2}^2+3\abs{z_3}^2)dz_1\wedge d\bar{z}_1+i(6x^2+\abs{z_2}^2+\abs{z_3}^2)dz_2\wedge d\bar{z}_2\\
    +i(6y^2+\abs{z_2}^2+\abs{z_3}^2)dz_3\wedge d\bar{z}_3+6ix(z_2dz_1\wedge d\bar{z}_2+\bar{z}_2dz_2\wedge d\bar{z}_1)\\
    +6iy(-iz_3dz_1\wedge d\bar{z}_3+i\bar{z}_3dz_3\wedge d\bar{z}_1)+i(\bar{z}_2z_3dz_2\wedge d\bar{z}_3+z_2\bar{z}_3dz_3\wedge d\bar{z}_2).
  \end{multline*}
  If we consider the $2\times 2$ block spanned by $dz_2$ and $dz_3$, we can see that this is positive definite unless $z_2=z_3=0$ and either $x=0$ or $y=0$.  Hence, the Levi form has at least two positive eigenvalues and satisfies $Z(2)$ except on this set.

  To define our metric, we let $\varphi_t(z)=t\abs{z_1}^2+\sum_{j=2}^4\abs{z_j}^2$ for some fixed $t\geq 1$ and use the K\"ahler form $\omega_t=i\ddbar\varphi_t$.  Note that $\omega_1$ is the Euclidean metric.
  We set $L_1^t=\frac{1}{\sqrt{t}}L_1$, $L_2^t=L_2$ and $L_3^t=L_3$.  If we use the notation $P^t_j=L^t_j P$, then in this basis our metric can be written
  $g^t_{j\bar k}= i\ddbar\varphi_t(i\bar L_k^t\wedge L^t_j) = \left<L_j^t,L_k^t\right>=\delta_{jk}+4P_j^t\bar{P}_k^t$.

  We write
  \[
    \Upsilon_t=i\sum_{j,k=1}^{n-1} b^{\bar k j}_t\bar{L}_k^t\wedge L_j^t.
  \]
  When $y=z_2=z_3=0$, we have $i\ddbar\rho=-36ix^2 dz_1\wedge d\bar{z}_1+6i x^2 dz_2\wedge d\bar{z}_2$,  $P^t_j=0$ if $j=2,3$ and $P^t_1 = -\frac{18}{\sqrt t}x^3$.
  Consequently, the eigenvalues of $\mathcal{L}$ are $\mu_1=\frac{-36x^2}{t(1+4\abs{P_1^t}^2)}=\frac{-36x^2}{t}+O(x^8)$,
  $\mu_2=0$, and $\mu_3=6x^2$.  To check condition \eqref{item:Levi form_property} of Definition \ref{defn:weak_Zq}, we compute
  \[
    \mu_1+\mu_2-\mathcal{L}(\Upsilon_t)=\frac{-36x^2}{t}(1-b^{\bar 1 1}_t)-6x^2b^{\bar 2 2}_t+O(x^8).
  \]
  Since $b^{\bar 2 2}_t\geq 0$ and $1\geq b^{\bar 1 1}_t+O(x^6)$, (by condition \eqref{item:upper_and_lower_bounds} of Definition \ref{defn:weak_Zq}) nonnegativity requires
  $b^{\bar 2 2}_t=0$ and $b^{\bar 1 1}_t=1+O(x^6)$.  Similar computations when $x=z_2=z_3=0$ require $b^{\bar 3 3}_t=0$ on this set.
  At the origin, $\Upsilon_t$ is now represented by a matrix whose eigenvalues are bounded between 0 and 1 (by condition \eqref{item:upper_and_lower_bounds} of Definition \ref{defn:weak_Zq} again) with diagonal entries of 0 and 1, so the off-diagonal entries must vanish.  Hence, $\Upsilon_t=i \bar{L}_1^t\wedge L_1^t$ at the origin.  Since nonvanishing terms of order $O(\abs{z})$ would cause the eigenvalues of $\Upsilon_t$ to grow larger than 1 or smaller than 0 in some direction, we conclude that $\Upsilon_t=i\bar{L}_1^t\wedge L_1^t+O(\abs{z}^2)$ near the origin.

  We note that $\mu_1+\mu_2-\mathcal{L}(\Upsilon_t)\geq 0$ (condition \eqref{item:Levi form_property} in Definition \ref{defn:weak_Zq}) is equivalent to $\Tr\mathcal{L}-\mathcal{L}(\Upsilon_t)\geq \mu_3$, which is in turn equivalent to $(\Tr\mathcal{L}-\mathcal{L}(\Upsilon_t))\omega_t-i\ddbar\rho\geq 0$ on $T^{1,1}(\partial\Omega)$.  Note that the inverse of our metric is
  \[
  g^{\bar k j}_t=\delta_{jk}-\frac{4}{1+4\abs{P^t}^2}P_k^t\bar{P}_j^t = \delta_{jk} - 4 P_k^t\bar P^t_j(1+O(|P^t|^2)).
  \]
  If $\{u_j\}$ is an orthonormal basis of $T^{1,0}(\p\Omega)$ near 0, then
  \begin{multline*}
    \Tr\mathcal{L} = \sum_{\ell=1}^3 \mathcal{L}(i\bar u_\ell \wedge u_\ell) = \sum_{j,k=1}^3 g^{\bar kj}_t \mathcal L (i \bar L^t_k \wedge L^t_j )
   = \frac{1}{t}(-36\abs{z_1}^2+3\abs{z_2}^2+3\abs{z_3}^2)g^{\bar 1 1}_t\\
   +6\abs{z_1}^2+2(\abs{z_2}^2+\abs{z_3}^2)+O((\abs{z_2}^2+\abs{z_3}^2)\abs{z}^6),
  \end{multline*}
  where
  \[
    g^{\bar 1 1}_t=\frac{1}{1+1296t^{-1}\abs{z_1}^6}+O((\abs{z_2}^2+\abs{z_3}^2)\abs{z}^4).
  \]
  The $g^{\bar 1 1}_t$ term is left intact in the computation because it does not lead to a term of the form $O((\abs{z_2}^2+\abs{z_3}^2)\abs{z}^6)$. In any case,
  \[
    \Tr\mathcal{L}-\mathcal{L}(\Upsilon_t)=6\abs{z_1}^2+2(\abs{z_2}^2+\abs{z_3}^2)+O(\abs{z}^4).
  \]
  If we test positivity of $((\Tr\mathcal{L}-\mathcal{L}(\Upsilon_t))\omega_t-i\ddbar\rho)(i\bar L_1^t\wedge L_1^t)$, we see that this is equivalent to
  \[
    6\abs{z_1}^2+2(\abs{z_2}^2+\abs{z_3}^2)-\frac{1}{t}(-36\abs{z_1}^2+3\abs{z_2}^2+3\abs{z_3}^2)+O(\abs{z}^4)\geq 0.
  \]
  Considering coefficients of $\abs{z_2}^2$, it is necessary that $2\geq\frac{3}{t}$, or $t\geq\frac{3}{2}$.  Since the Euclidean metric corresponds to $t=1$, weak $Z(2)$ will fail for the Euclidean metric.

  For the positive result, we set $\Upsilon_t=i(g_{1\bar 1}^t)^{-1}\bar{L}_1^t\wedge L_1^t$, so that conditions \eqref{item:upper_and_lower_bounds} and \eqref{item:trace} are immediately satisfied.  Note that $(g_{1\bar 1}^t)^{-1}=g^{\bar 1 1}_t+O((\abs{z_2}^2+\abs{z_3}^2)\abs{z}^4)$, so we have
  \[
    \Tr\mathcal{L}-\mathcal{L}(\Upsilon_t)=6\abs{z_1}^2+2(\abs{z_2}^2+\abs{z_3}^2)+O((\abs{z_2}^2+\abs{z_3}^2)\abs{z}^6).
  \]
  For $\abs{z}$ sufficiently small, this gives us
  \[
    \Tr\mathcal{L}-\mathcal{L}(\Upsilon_t)\geq 6\abs{z_1}^2+\frac{3}{2}(\abs{z_2}^2+\abs{z_3}^2).
  \]
  Hence,
	  \begin{multline*}
    (\Tr\mathcal{L}-\mathcal{L}(\Upsilon_t))\omega_t-i\ddbar\rho\geq
    i\left(6(t+6)\abs{z_1}^2+\frac{3}{2}(t-2)(\abs{z_2}^2+\abs{z_3}^2)\right)dz_1\wedge d\bar{z}_1\\+i\left(6y^2+\frac{1}{2}(\abs{z_2}^2+\abs{z_3}^2)\right)dz_2\wedge d\bar{z}_2
    +i\left(6x^2+\frac{1}{2}(\abs{z_2}^2+\abs{z_3}^2)\right)dz_3\wedge d\bar{z}_3\\-6ix(z_2dz_1\wedge d\bar{z}_2+\bar{z}_2dz_2\wedge d\bar{z}_1)
    -6iy(-iz_3dz_1\wedge d\bar{z}_3+i\bar{z}_3dz_3\wedge d\bar{z}_1)\\-i(\bar{z}_2z_3dz_2\wedge d\bar{z}_3+z_2\bar{z}_3dz_3\wedge d\bar{z}_2).
  \end{multline*}
  To check that this is positive for $t\geq 6$, we first evaluate the determinant of the $2\times 2$ minor spanned by $dz_2$ and $dz_3$:
  \begin{multline*}
    \left(6y^2+\frac{1}{2}(\abs{z_2}^2+\abs{z_3}^2)\right)\left(6x^2+\frac{1}{2}(\abs{z_2}^2+\abs{z_3}^2)\right)-\abs{z_2}^2\abs{z_3}^2=\\
    36 x^2y^2+3\abs{z_1}^2(\abs{z_2}^2+\abs{z_3}^2)+\frac{1}{4}(\abs{z_2}^2-\abs{z_3}^2)^2.
  \end{multline*}
  Since this determinant is nonnegative and the terms of the diagonal are nonnegative, it follows from
  \cite[Theorem 4.3.8]{HoJo85} that the form under consideration must have at least two nonnegative eigenvalues.  Noting that the determinant of this minor is bounded below by
  $3\abs{z_1}^2(\abs{z_2}^2+\abs{z_3}^2)$, we can bound the determinant of the entire $3\times 3$ form below by
  \begin{multline*}
    3\abs{z_1}^2(\abs{z_2}^2+\abs{z_3}^2)\left(6(t+6)\abs{z_1}^2+\frac{3}{2}(t-2)(\abs{z_2}^2+\abs{z_3}^2)\right)\\
    -36y^2\abs{z_3}^2\left(6y^2+\frac{1}{2}(\abs{z_2}^2+\abs{z_3}^2)\right)-36x^2\abs{z_2}^2\left(6x^2+\frac{1}{2}(\abs{z_2}^2+\abs{z_3}^2)\right).
  \end{multline*}
  Substituting $t\geq 6$ into this we can check that the positive terms dominate the negative terms, so the form is positive and weak $Z(2)$ is satisfied for $\abs{z}$ sufficiently small.  Lemma \ref{lem:bounded_domain} can now be used to create a bounded domain where $\abs{z}$ remains sufficiently small for weak $Z(2)$ to hold on the entire domain.
\end{proof}

\bibliographystyle{alpha}
\bibliography{mybib3}
\end{document}